\documentclass[12pt]{amsart}

\usepackage{amsmath, amsthm, graphics, setspace}
\usepackage{latexsym}
\usepackage{mathtools}
\usepackage{eucal}
\usepackage{caption}
\usepackage{verbatim}
\usepackage[utf8]{inputenc}
\usepackage{amssymb}
\usepackage{color}
\usepackage{amscd}
\usepackage{enumitem}
\usepackage{palatino}
\usepackage{latexsym}
\usepackage{caption}
\usepackage{epsfig}
\usepackage{graphicx}
\usepackage{amsfonts}
\usepackage{psfrag}
\usepackage[all]{xy}
\usepackage{youngtab, ytableau}

\usepackage{tikz}
\usetikzlibrary{backgrounds, fit, matrix}
\usetikzlibrary{positioning}
\usetikzlibrary{calc,through,chains}
\usetikzlibrary{arrows,shapes,snakes,automata, petri}

\usepackage{url}
\usepackage{cite}

\usepackage[margin=1.3in]{geometry}

\definecolor{ltgrey}{RGB}{180, 187, 198}

\input xy
\xyoption{all}
\UseComputerModernTips

\numberwithin{equation}{section}
\numberwithin{figure}{section}

\newtheorem{THM}{Theorem}

\newtheorem{COR}[THM]{Corollary}

\newtheorem{theorem}{Theorem}[section]
\newtheorem{lemma}[theorem]{Lemma}
\newtheorem{proposition}[theorem]{Proposition}

\newtheorem{corollary}[theorem]{Corollary}

\newtheorem{remark}[theorem]{Remark}

\newtheorem{example}[theorem]{Example}

\theoremstyle{definition}
\newtheorem{definition}[theorem]{Definition}

\newcommand{\C}{{\mathbb{C}}}

\newcommand{\Z}{{\mathbb{Z}}}

\newcommand{\N}{{\mathbb{N}}}

\newcommand{\B}{\mathcal{B}}

\newcommand{\p}{\mathfrak{p}}

\def\invv{\mathsf{i}\mathsf{n}\mathsf{v}}

\def\invv{{\mathsf i}{\mathsf n}{\mathsf v}}

\usepackage{multicol}
\usepackage{color} 

\definecolor{gold}{rgb}{0.85,.66,0}
\definecolor{cherry}{rgb}{0.9,.1,.2}
\definecolor{burgundy}{rgb}{0.8,.2,.2}
\definecolor{orangered}{rgb}{0.85,.3,0}
\definecolor{orange}{rgb}{0.85,.4,0}
\definecolor{olive}{rgb}{.45,.4,0}
\definecolor{lime}{rgb}{.6,.9,0}
\definecolor{green}{rgb}{.2,.7,0}
\definecolor{grey}{rgb}{.4,.4,.2}
\definecolor{brown}{rgb}{.4,.3,.1}

\def\m{{\mathbf m}}

\def\S{{\mathbf{S}}}

\def\cat{\mathsf{Cat}}

\newcommand{\Hess}{\mathrm{Hess}}

\newcommand{\Clan}{\mathbf{Clan}}

\newcommand{\Supp}{\mathrm{Supp}}

\newcommand{\area}{\mathrm{area}}

\def\x{\mathsf{x}}

\def\ov{\overline}

\def\O{{\mathcal O}}

\newcommand{\mf}[1]{\mathfrak{#1}}

\begin{document}

\title{$K$-Orbit closures and Hessenberg varieties}

\author{Mahir Bilen Can}
\address{Department of Mathematics\\ Tulane University \\ 6823 St. Charles Avenue \\ New Orleans, Louisiana 70118 \\ U.S.A. }
\email{mahirbilencan@gmail.com}

\author{Martha Precup}
\address{Department of Mathematics\\ Washington University in St. Louis \\ One Brookings Drive \\ St. Louis, Missouri  63130 \\ U.S.A. }
\email{martha.precup@wustl.edu}

\author{John Shareshian}
\address{Department of Mathematics\\ Washington University in St. Louis \\ One Brookings Drive \\ St. Louis, Missouri  63130 \\ U.S.A. }
\email{jshareshian@wustl.edu}

\author{\"Ozlem U\u{g}urlu}
\address{Department of Mathematics and Statistics\\ Saint Louis University \\ 220 N. Grand Blvd\\ St. Louis, Missouri  63103\\ U.S.A. }
\email{ozlem.ugurlu@slu.edu}

\maketitle

\begin{abstract}
This article explores the relationship between Hessenberg varieties associated with semisimple operators with two eigenvalues and orbit closures of a spherical subgroup of the general linear group.  
    We establish the specific conditions under which these semisimple Hessenberg varieties are irreducible.  We determine the dimension of each irreducible Hessenberg variety under consideration and show that the number of such varieties is a Catalan number. We then apply a theorem of Brion to compute a polynomial representative for the cohomology class of each such variety.
    Additionally, we calculate the intersections of a standard (Schubert) hyperplane section of the flag variety with each of our Hessenberg varieties and prove this intersection possess a cohomological multiplicity-free property.
\end{abstract}

\vspace{.5cm}
\noindent \textbf{Keywords.} Hessenberg varieties, symmetric varieties, involutions, Catalan numbers, Monk's formula

\smallskip
\noindent \textbf{MSC.} 14M15, 14M27, 05A05

\tableofcontents
\section{Introduction}{\label{S:Introduction}}

Let $n$ be a positive integer and let $G=GL_n(\C)$.  Given positive integers $p,q$ such that $p+q=n$, let $K$ be a Levi subgroup of the stabilizer in $G$ of a $p$-dimensional subspace of $\C^n$.  So, $K \cong GL_p(\C) \times GL_q(\C)$.  Then $K$ is spherical.
We examine coincidences between two well-studied classes of subvarieties in the type A flag variety: Hessenberg varieties and $K$-orbit closures. We identify a collection of Hessenberg varieties, each equal to the closure of a single $K$-orbit.
Leveraging the theory of $K$-orbits we answer, for this particular collection, questions that are difficult to settle for arbitrary Hessenberg varieties.

Let $B$ be the Borel subgroup of $G$ consisting of upper triangular matrices. The flag variety $\B=G/B$ has been studied extensively.  More recently, Hessenberg varieties, which were first studied due to their connection with numerical linear algebra, have been of interest to geometers, representation theorists, and combinatorialists.  

We identify $\B$ with the collection of full flags
$$
V_\bullet=0<V_1<\ldots<V_{n-1}<V_n=\C^n
$$
with $\dim V_i=i$ for all $i \in [n]:=\{1,\dots, n\}$. A {\it Hessenberg vector} is a weakly increasing sequence $\m=(m_1,\ldots,m_n)$ of integers satisfying $i \leq m_i \leq n$ for each $i \in [n]$.  Given such $\m$ and any $n \times n$ matrix $\x$, the associated {\it Hessenberg variety} is
$$
\Hess(\x,\m):=\{V_\bullet \in \B \mid \x V_i \leq V_{m_i} \mbox{ for all } i \in [n]\}.
$$
While there have been more recent developments, the survey \cite{AbeHoriguchi2020} by Abe and Horiguchi gives a nice summary of the work on Hessenberg varieties and connections to various fields.

Despite their elementary definition, some basic questions about the structure of Hessenberg varieties remain wide open.  The ones of interest herein follow.
\begin{enumerate}[label=(\Alph*)]
    \item What is the dimension of $\Hess(\x,\m)$?
    \item For which matrices $\x$ and Hessenberg vectors $\m$ is $\Hess(\x,\m)$ irreducible?
    \item If $\Hess(\x,\m)$ is irreducible, can we describe cohomology class in $H^\ast(\B; \Z)$ it represents?
\end{enumerate}

Let us give an example illustrating that Questions (A) and (B) are subtle, in that their answers can depend on the choice of matrix $\x$ when $\m$ is fixed.  

\begin{example} \label{ex.23n}
Consider the Hessenberg vector $\m=(2,3,4,\ldots,n,n)$.  If ${\mathsf s}$ is a regular semisimple matrix, then by work of De Mari, Procesi, and Shayman in~\cite{demariprocesishayman}, $\Hess({\mathsf s},\m)$ is isomorphic to the toric variety associated to the fan of type $A_{n-1}$ Weyl chambers.  In particular, $\Hess({\mathsf s},\m)$ is irreducible of dimension $n-1$.  

For $i \in [n-1]$, let $w^i \in \S_n$ be the unique permutation satisfying
\begin{itemize}
\item $w^i(1)=i+1$,
\item $w^i(n)=i$, and
\item $w^i(j)>w^i(j+1)$ for $2 \leq j \leq n-2$.
\end{itemize}
We write $E_{1n}$ for the $n \times n$ elementary matrix whose only nonzero entry is in its first row and last column.  As shown by Tymoczko in \cite{Tymoczko2006b}, $\Hess(E_{1n},\m)$ is the union of the Schubert varieties $X_{w^i}$, from which it follows that $\Hess(E_{1n},\m)$ has $n-1$ irreducible components, each of dimension $1+{{n-1} \choose {2}}.$
\end{example}
We remark that for a fixed Hessenberg vector $\m$ there can be irreducible varieties $\Hess(\x,\m)$ and $\Hess({\mathsf y},\m)$ of differing dimensions.  For example, if $m_1<n$ and $m_j=n$ for $j>1$, then $\Hess(\x,\m)=\B$ if and only if $\x$ is scalar, while $\Hess({\mathsf y},\m)$ is irreducible of dimension $\dim(\B) - (n-m_1)$ whenever ${\mathsf y}$ is regular.

The results on $\Hess(E_{1n},(2,3,\ldots,n,n))$ discussed in Example~\ref{ex.23n} are worth further consideration.  The key point is that for each $g \in B$, $E_{1n}g=\lambda gE_{1n}$ for some $\lambda \in \C$.  That is, the Borel subgroup $B$ stabilizes the subspace spanned by $ E_{1n}$ under the adjoint action.  It follows directly that for every Hessenberg vector $\m$, $\Hess(E_{1n},\m)$ is $B$-invariant and therefore a union of $B$-orbits.  Thus every irreducible component of $\Hess(E_{1n},\m)$ is a Schubert variety $X_w$ for some $w \in \S_n$.  One can determine which $X_w$ appears as such components for any given $\m$; see~\cite{Tymoczko2006b, Abe-Crooks}.

We use the approach described in the previous paragraph to study $\Hess(\x,\m)$ when $\x$ is semisimple with exactly two distinct eigenvalues.  Given such $\x$ with eigenvalues $\lambda,\mu$ of respective multiplicities $p,q$ (hence $p+q=n$), let $Y,Z$ be the associated eigenspaces.  Thus $\C^n=Y \oplus Z$.  The simultaneous stabilizer $K$ of $Y$ and $Z$ in $G$ is isomorphic to $GL_p(\C) \times GL_q(\C)$, and it is straightforward to see that $\Hess(\x,\m)$ is a union of $K$-orbits. It is well-known (see for example \cite{Wolf1969}) that $K$ is {\it spherical}, that is, $K$ has finitely many orbits on $\B$.  We will use the classification and theory of $K$-orbits on $\B$ due to Yamamoto~\cite{Yamamoto} and many others~\cite{Matsuki-Oshima, Wyser, CanUgurlu} to address Questions (A), (B), (C) above for Hessenberg varieties defined using such $\x$.

Assume as above that the semisimple matrix $\x$ has exactly two distinct eigenvalues $\lambda,\mu$ with respective multiplicities $p,q$ and fix a Hessenberg vector $\m$.  We observe that the isomorphism type of $\Hess(\x,\m)$ depends only on $p$ and $q$.  Indeed, since $\Hess(g^{-1}\x g,\m)=g\Hess(\x,\m)$ for every $g \in G$, we may assume that $\x=(x_{ij})$ is diagonal with $x_{ii}=\lambda$ for $i \in [p]$ and $x_{ii}=\mu$ for $p<i \leq n$.  Moreover, it is straightforward to show that for scalars $\alpha \neq 0$ and $\beta$,
$$
\Hess(\alpha\x+\beta I,\m)=\Hess(\x,\m)
$$
hence $\lambda$ and $\mu$ are irrelevant and our observation follows.  So, there is no harm in writing $\x_{p,q}$ to denote any such semisimple matrix $\x$.

We summarize now our results on $\Hess(\x_{p,q},\m)$.  Our first result addresses Question (B).

\begin{THM}[See Corollaries \ref{cor.xpqorbitclosure} and \ref{cor.312free} and Theorem \ref{thm.irreducible} below] \label{thm1.intro} The following conditions on the Hessenberg variety $\Hess(\x_{p,q},\m)$ are equivalent.
\begin{enumerate}
\item $\Hess(\x_{p,q},\m)$ is irreducible.
\item There is a Hessenberg vector $(\ell_1,\ldots,\ell_q)$ of length $q$ such that $m_i=\ell_i+p$ for $i \leq q$ and $m_i=n$ for $q<i \leq n$.
\item $\Hess(\x_{p,q}, \m)$ is the closure of one of $\frac{1}{q+1}{{2q} \choose {q}}$ orbits of $K$ on $\B$.  This collection of orbits is 
naturally parameterized by $231$-free permutations in~$\S_q$.
\end{enumerate}
\end{THM}

There is a formula for the dimensions of $K$-orbits in a flag variety (see \cite[Section 2.3]{Yamamoto}). This formula allows us to compute and write a nice formula for the dimension of any irreducible $\Hess(\x_{p,q},\m)$, thereby addressing Question (A) for this collection.

\begin{COR}[See Corollary~\ref{cor.dimension} below] \label{cor.intro} If $\m=(m_1,\ldots,m_n)$ is a Hessenberg vector such that $\Hess(\x_{p,q},\m)$ is irreducible, then
$$
\dim\Hess(\x_{p,q},\m)=\sum_{i=1}^n(m_i-i).
$$
\end{COR}

Previous work on the Question (C) addresses the case where $\x$ is regular.  
It is known that the class of any regular Hessenberg variety depends only on the underlying Hessenberg vector~\cite{Abe-Fujita-Zeng}. Polynomial representatives for the classes of regular Hessenberg varieties were first identified as specializations of certain double Schubert polynomials~\cite{Anderson-Tymoczko, Insko-Tymoczko-Woo}. Even more recently, Nadeau and Tewari~\cite{NadeauTewari2021} gave a combinatorial formula expressing each as a sum of Schubert polynomials.
Here we consider certain cases in which $\x$ is not regular.

Let us state a more specific version of Question (C).  The cohomology classes associated with the Schubert varieties $X_w$ ($w \in \S_n$) form a basis for $H^\ast(\B;\Z)$.  Let $I$ be the ideal in $R:=\Z[x_1,\ldots,x_n]$ generated by constant-free symmetric polynomials.  There is an isomorphism $\phi$ from $H^\ast(\B;\Z)$ to $R/I$ mapping the class associated to $X_w$ to the Schubert polynomial $\mf{S}_w$. (This presentation of $H^\ast(\B;\Z)$ is due to Borel; see \cite{Borel1953} or \cite{Manivel}.)  Given any irreducible subvariety ${\mathcal V}$ of $\B$, one can ask how to expand the image $\mf{S}({\mathcal V})$ under $\phi$ of the class associated to ${\mathcal V}$ as a linear combination of Schubert polynomials.  We obtain the following result for the collection of irreducible Hessenberg varieties introduced in the statement of Theorem~\ref{thm1.intro}.

\begin{THM}[See Corollary~\ref{cor.Hess.class}] \label{thm2.intro} 
Let $X:=\Hess(\x_{p,q},\m)$ be an irreducible Hessenberg variety indexed by a $231$-free permutation $w\in \S_q$. 
A polynomial representative of the class $\mathfrak{S}(X)$ of $\Hess(\x_{p,q},\m)$ in the integral cohomology ring of the flag variety is given by the following sum of Schubert polynomials 
    \begin{align*}
    \mathfrak{S}(X)=\sum_{(u,v)} \mf{S}_{uw_0v^{-1}w_0}, 
     \end{align*}
     where the sum is taken over all pairs $(u,v)\in \S_q\times \S_q$ such that $wy_0 = uv$ and $\ell(wy_0) = \ell(u)+\ell(v)$.
\end{THM}

A key ingredient in our computations for Theorem~\ref{thm2.intro} is the useful notion of the $W$-set associated with a $K$-orbit $\mathcal{O}=KV_\bullet$ in the flag variety. 
Loosely speaking, the $W$-set of $\mathcal{O}$ consists of permutations that are obtained by multiplying the simple reflections that label the edges of certain saturated paths in the weak order on the spherical variety $G/K$; see Section~\ref{ss.weak-order} for more. 
The origins of $W$-sets go back to the influential work of Richardson and Springer in~\cite{RichardsonSpringer}, where the authors initiated a systematic study of the (weak) Bruhat orders on the Borel orbit closures in symmetric varieties.
This development is generalized by Knop to all spherical homogeneous varieties in~\cite{Knop1995}. 
Brion's work~\cite{Brion2001} has brought to light a multitude of fascinating applications of $W$-sets to the geometry of $K$-orbits. 
In particular, Brion used $W$-sets to describe certain deformations of $K$-orbits in flag varieties to the unions of Schubert varieties; the results of Theorem~\ref{thm2.intro} rest heavily on this work.
More recently, combinatorialists have used $W$-sets to develop Schubert calculus for (classical) symmetric spaces. 
There is currently a fast-growing literature on this subject~\cite{WyserYong2014, WyserYong2017, HMP2018, HMP2021, HMP2022}.

It follows directly from Theorem \ref{thm2.intro} that if $\Hess(\x_{p,q}, \m)$ is irreducible, then the polynomial $\mf{S}(\Hess(\x_{p,q},\m))$ is a $0-1$ sum of Schubert polynomials. In other words, when we express $\mf{S}(\Hess(\x_{p,q},\m))$ as a linear combination of Schubert polynomials, all coefficients lie in $\{0,1\}$. Whenever a polynomial is a $0-1$ sum of Schubert polynomials, we say that the sum is {\em multiplicity-free}. Something stronger is true.  For $i \in [n-1]$, we write $s_i$ for the transposition $(i,i+1) \in \S_n$.

\begin{THM}[See Theorem \ref{thm.monkmultfree} below] \label{thm3.intro}
If $i \in [n-1]$ and $\Hess(\x_{p,q}, \m)$ is irreducible, then the product $\mf{S}_{s_i}\mf{S}(\Hess(\x_{p,q},\m))$ is a multiplicity-free sum of Schubert polynomials.
\end{THM}

Theorem \ref{thm3.intro}, which is a consequence of Theorem \ref{thm2.intro} and Monk's formula, gives insight into how $\Hess(\x_{p,q},\m)$ intersects certain Schubert varieties of codimension one in $\B$.

Geometrically speaking, at the cycle level, the classical Monk's formula (\cite[Theorem 3]{Monk1959}) says that the intersection of a Schubert variety $X\subseteq G/B$ with a Schubert divisor $Z \subset G/B$ is a multiplicity-free sum of Schubert divisors of $X$. Although $\Hess(\x_{p,q},\m)$ has a flat degeneration to a union $Y$ of (many) Schubert varieties, it is not a $B$-stable subvariety of $G/B$.
In light of this fact, we find it rather surprising that the cohomology class of the intersection of $\Hess(\x_{p,q},\m)$ with $Z$ is a $0-1$ sum of the classes of Schubert divisors in $Y$. It is unknown to us that if this multiplicity-free phenomenon persists in all cases of the intersection between $Z$ and any $K$-orbit closure or any irreducible semisimple Hessenberg variety in the flag variety.

It is natural to ask whether the methods used here and illustrated in Example \ref{ex.23n} are more widely applicable.  
The key idea is that if $\Hess(\x,\m)$ is invariant under the action of a spherical group $H$, then known combinatorial descriptions of $H$-orbits allow for a detailed analysis of $\Hess(\x,\m)$ that is difficult to carry out for arbitrary Hessenberg varieties. 
If a spherical subgroup $H$ of $G$ centralizes $\x$ (up to multiplication by a scalar) then $H$ indeed acts on $\Hess(\x,\m)$ for all $\m$.  However, this situation is rare.  If $\x$ is semisimple, then $C_G(\x)$ is reductive.  The reductive spherical subgroups of $G$ are known (see~\cite{Kramer1979},\cite{Brion1987},\cite{Mikityuk1986}).  These are the centralizers of the matrices $\x_{p,q}$ studied herein along with the classical groups that act irreducibly on $\C^n$.  In the second case, the centralizer of every such classical group consists of the scalar matrices, and if $\x$ is scalar then $\Hess(\x,\m)=\B$ for all $\m$.  There are nilpotent matrices other than conjugates of $E_{1n}$ with spherical centralizers in $G$, but these are also rare. The automorphism group of $\Hess(\x,\m)$ can be much larger than $C_G(\x)$, but it seems challenging to give a comprehensive and useful analysis of this phenomenon.  On the other hand, in \cite{demariprocesishayman}, De Mari, Procesi, and Shayman define Hessenberg varieties for arbitrary reductive groups. In Lie types other than $A$ there are additional examples of reductive spherical subgroups centralizing nonscalar elements. 
We will examine these examples in future work.

The content of the rest of the paper is as follows. After reviewing the requisite results in Section \ref{S:Notation}, we prove Theorem \ref{thm1.intro} and Corollary \ref{cor.intro} in Section~\ref{S:Irreducible}. The proofs of Theorems \ref{thm2.intro} and \ref{thm3.intro} are the subject of Section~\ref{S:Wsets}.

\

\noindent\textbf{Acknowledgements.}
The first author is partially supported by a grant from the Louisiana Board of Regents (contract no. LEQSF(2021-22)-ENH-DE-26).
The second author is partially supported by NSF Grant DMS 1954001.

\section{Notation and preliminaries}{\label{S:Notation}}

We review here various results and definitions that we will use below. We denote by $\Z_+$ the set of positive integers.
Let $n\in \Z_+$. Let $G=GL_n(\C)$ and let $B \leq G$ be the Borel subgroup consisting of upper triangular matrices.  The flag variety $G/B$ will be denoted by $\B$.  We identify each coset $gB \in \B$ with the flag
$$
V_\bullet=0<V_1<\ldots <V_{n-1}<V_n=\C^n
$$
in which each $V_i$ is spanned by the first $i$ columns of $g$.

Denote the symmetric group on $[n]$ by $\S_n$. Let $p$ and $q$ be positive integers such that $n=p+q$. We frequently consider the smaller symmetric group $\S_q$ below, which we identify with the subgroup of $\S_n$ stabilizing $[n] \setminus [q]$ pointwise.  For $i \in [n-1]$, we write $s_i$ for the simple reflection $(i,i+1) \in \S_n$.  A {\it reduced word} for $w \in \S_n$ is any shortest possible representation
$$
w=s_{i_1}s_{i_2} \ldots s_{i_\ell}
$$
of $w$ as a product of simple reflections.  We call the set of simple transpositions that appear in any reduced expression of $w$ the \textit{support of $w$} and denote it by $\Supp(w)$. For example, $\Supp(2143) = \Supp(s_1s_3) = \{s_1, s_3\}$.

The {\it length} $\ell(w)$ of $w\in \S_n$ is the number of simple reflections appearing in any reduced word for $w$. It is well known that $$\ell(w)=\mid\{i<j \mid 1\leq i<j \leq n,\, w(i)>w(j)\}\mid$$
for all $w\in \S_n$.
The longest elements of both $\S_n$ and $\S_q$ play a role below; to avoid confusion, we write $w_0$ for the longest element of $\S_n$ and $y_0$ for the longest element of $\S_q$.

We say that $w \in \S_n$ {\it avoids $312$} (or is {\it $312$-free}) if there do not exist $1 \leq i<j<k \leq n$ such that $w(j)<w(k)<w(i)$ and define avoidance of $231$ similarly. It is straightforward to show that $w$ avoids $231$ if and only if $w^{-1}$ avoids $312$.

\subsection{Hessenberg varieties}

A \textit{Hessenberg vector} is a weakly increasing sequence
$$
\m=(m_1,\ldots,m_n)
$$
of integers satisfying $i \leq m_i \leq n$ for each $i \in [n]$.  Given a matrix $\x\in \mathfrak{g}:=\mathfrak{gl}_n(\C)$ and Hessenberg vector $\m$ we define the corresponding \textit{Hessenberg variety} by
$$
\Hess(\x,\m) :=\{V_\bullet \in \B\mid \x V_i \leq V_{m_i} \mbox{ for all } i \in [n]\}.
$$
Given a Hessenberg vector $\m$ we define $\pi_\m$ to be the lattice path from the upper left corner to the lower right corner of an $n \times n$ grid in which the vertical step in row $i$ occurs in column $m_i$.  Since $m_i \geq i$, $\pi_\m$ is a \textit{Dyck path}, that is, the lattice path $\pi_\m$ never crosses the diagonal connecting the two corners.  We write $\area(\pi_\m)$ for the number of squares in the grid that lie below $\pi_m$ and strictly above the diagonal and observe that
$$
\area(\pi_\m)=\sum_{i=1}^n(m_i-i).
$$
Herein we examine Hessenberg varieties $\Hess(\x_{p,q},\m)$ where $\x_{\p,q} \in \mathfrak{g}$ is semisimple with exactly two distinct eigenvalues, one of multiplicity $p$ and one of multiplicity $q$ (so $p+q=n$).  Since $\Hess(g^{-1}\x g,\m)=g\Hess(\x,m)$ for all $g \in G$ and all $\x \in \mathfrak{g}$, we assume without loss of generality that
\begin{align}\label{A:semisimple}
\x_{p,q} = \text{diag}(\underbrace{\lambda_1,\dots, \lambda_1}_{\text{$p$ times}},
\underbrace{\lambda_2,\dots, \lambda_2}_{\text{$q$ times}}),
\end{align}
for distinct $\lambda_1,\lambda_2 \in \C$. 

The centralizer of $\x_{p,q}$ in $G$ is the subgroup $K \cong GL_p(\C)\times GL_q(\C)$ consisting of all $g=(g_{ij}) \in G$ such that $g_{ij}=0$ if either $i \leq p<j$ or $j \leq p<i$. It is staightforward to confirm that if $V_\bullet \in \Hess(\x_{p,q},\m)$ and $g \in K$, then
$$gV_\bullet:=0<gV_1<\ldots <gV_{n-1}<\C^n \in \Hess(\x_{p,q},\m).$$
Thus $\Hess(\x_{p,q},\m)$ is a union of $K$-orbits on $\B$.

\subsection{$K$-orbits on the flag variety}
The group $K$ is known to have finitely many orbits on the flag variety $\B$. 
These orbits are parameterized by combinatorial objects called clans.   Clans originated in work of Matsuki and \={O}shima~\cite{Matsuki-Oshima} 
to parameterize symmetric subgroup orbits on complex flag manifolds of classical type. Their notation has morphed with developments through
subsequent works, notably by Yamamoto~\cite{Yamamoto} and then Wyser~\cite{Wyser}.

We define the set of clans as follows. Consider the set of all sequences
$$
\gamma=c_1c_2\cdots c_n
$$
such that
\begin{enumerate}
\item each $c_i$ lies in $\{+,-\} \cup \Z_+$,
\item each element of $\Z_+$ appearing in $\gamma$ appears exactly twice, and
\item if $+$ and $-$ appear, respectively, exactly $s$ times and $t$ times in $\gamma$, then $s-t=p-q$. 
\end{enumerate}
We define an equivalence relation on this set by identifying sequences $\gamma=c_1\ldots c_n$ and $\delta=d_1\ldots d_n$ if
\begin{itemize}
    \item $d_i=d_j \in \Z_+$ whenever $c_i=c_j \in \Z_+$, and
    \item $d_i=c_i$ whenever $c_i \in \{+,-\}.$
\end{itemize}
A {\it $(p,q)$-clan} (or {\it clan} if $p,q$ are fixed) is an equivalence class of this relation.  We identify a clan with its unique representative $\gamma$ satisfying
\begin{itemize}
    \item if $j>1 \in \Z_+$ appears in $\gamma$ then $j-1$ appears in $\gamma$ and the first occurrence of $j-1$ is to the left of the first occurrence of $j$,
\end{itemize}
and write $\Clan_{p,q}$ for the set of all such representatives.  So, for example, $5++3-+35+$ and $1++2-+21+$ lie in the same $(6,3)$-clan and the second of these is our fixed representative for the equivalence class.  In general, if $\gamma \in \Clan_{p,q}$ then there is some $\ell \in \Z_{\geq 0}$ such that the integers appearing in $\gamma$ are exactly those in $[\ell]$, and if $s$ entries of $\gamma$ are plus signs and $t$ entries are minus signs, then $p=\ell+s$ and $q=\ell+t$.

A flag $V(\gamma)_\bullet$ in $\B$ is associated with each clan $\gamma$ in the next definition.

\begin{definition}\label{def.flag.rep} Let $e_1,\ldots,e_n$ be the standard basis for $\C^n$. Given $(p,q)$-clan $\gamma=c_1\ldots c_n$, define $v_1,\ldots,v_n \in \C^n$ as follows.
\begin{itemize}
\item If $c_i$ is the $k^{th}$ occurrence of $+$ in $\gamma$ and exactly $\ell$ elements of $[q]$ have appeared twice among $c_1, \ldots, c_{i-1}$, set $v_i=e_{k+\ell}$.
\item If $c_i$ is the $k^{th}$ occurrence of $-$ in $\gamma$ and exactly $\ell$ elements of $[q]$ have appeared twice among $c_1, \ldots, c_{i-1}$, set $v_i=e_{p+k+\ell}$.
\item Say $c_i=c_j=k \in [q]$ for some $i<j$, with exactly $r$ occurrences of $+$ appearing in $c_1 \cdots c_{i-1}$, exactly $s$ occurrences of $-$ appearing in $c_1 \cdots c_{j-1}$, and exactly $u$ elements of $[q]$ appearing twice in $c_1\cdots c_{j}$.  Then set $v_i=e_{k+r}+e_{p+s+u}$ and $v_j=e_{k+r}-e_{p+s+u}$.
\end{itemize}
For $i \in [n]$, set $$V(\gamma)_i:=\C\{v_j \mid j \leq i\}$$ and define $$V(\gamma)_\bullet:=0<V(\gamma)_1< \ldots < V(\gamma)_{n-1}<\C^n \in \B.$$
\end{definition}

We observe that, for arbitrary $\gamma$, each vector $v_i$ used to construct $V(\gamma)_\bullet$ is either a standard basis vector or of the form $e_r \pm e_s$ with $r \in [p]$ and $p<s \leq n$.

\begin{example} \label{ex.flag}
Say $p=5$, $q=3$, and $\gamma=+1+-2+21$.  Then $v_1=e_1$, $v_2=e_2+e_8$, $v_3=e_3$, $v_4=e_6$, $v_5=e_4+e_7$, $v_6=e_5$, $v_7=e_4-e_7$, and $v_8=e_2-e_8$.
\end{example}

\begin{definition} \label{def.k.orbit.clan}
Given a $(p,q)$-clan $\gamma$, we set
$$\O_\gamma:=KV(\gamma)_\bullet,$$
so $\O_\gamma$ is the $K$-orbit on $\B$ containing $V(\gamma)_\bullet$.
\end{definition}

\begin{lemma}[Matsuki--\={O}shima] \label{lemma.flag.rep}
Each $K$-orbit on $\B$ contains a unique flag $V(\gamma)_\bullet$, therefore each $K$-orbit on $\B$ is of the form $\O_\gamma$ for some $\gamma\in \Clan_{p,q}$. Furthermore, $\O_\gamma=\O_\delta$ for $\gamma,\delta\in \Clan_{p,q}$ if and only if $\gamma=\delta$.
\end{lemma}

\begin{definition} \label{def.inclusion.order} Given $\gamma, \tau\in \Clan_{p,q}$ we write $\gamma\leq \tau$ whenever $\O_\gamma \subseteq \overline{\O_\tau}$. We call the partial order $\leq$ the \textit{inclusion order} on $\Clan_{p,q}$. 
\end{definition}

We now present a result of Wyser~\cite{Wyser} characterizing the inclusion order.
Given a clan $\gamma=c_1c_2\cdots c_n$, we define
\begin{enumerate}
\item $\gamma(i;+)$ to be the  total number of plus signs and pairs of equal natural numbers occurring among $c_1\cdots c_i$,
\item $\gamma(i;-)$ to be the  total number of minus signs and pairs of equal natural numbers occurring among $c_1\cdots c_i$, and
\item $\gamma(i,j)$ to be the number of pairs of equal numbers $c_s=c_t\in \Z_+$ with $s\leq i<j<t$.
\end{enumerate}

\begin{example} \label{e2}
If $\gamma=+1+-2+21$ as in Example \ref{ex.flag} above, then
$$
(\gamma(i;+))_{i=1}^n=(1,1,2,2,2,3,4,5),
$$
$$
(\gamma(i;-))_{i=1}^n=(0,0,0,1,1,1,2,3),
$$
and
$$
(\gamma(i,j))_{i,j=1}^n=\left( \begin{array}{cccccccc} 0 & 0 & 0 & 0 & 0 & 0 & 0 & 0 \\ 0 & 0 & 1 & 1 & 1 & 1 & 1 & 0 \\ 0 & 0 & 0 & 1 & 1 & 1 & 1 & 0 \\ 0 & 0 & 0 & 0 & 1 & 1 & 1 & 0 \\ 0 & 0 & 0 & 0 & 0 & 2 & 1 & 0 \\ 0 & 0 & 0 & 0 & 0 & 0 & 1 & 0 \\ 0 & 0 & 0 & 0 & 0 & 0 & 0 & 0 \\ 0 & 0 & 0 & 0 & 0 & 0 & 0 & 0 \end{array} \right).
$$ 
\end{example}

\begin{theorem}[Wyser] \label{thm.inclusion} Let $\gamma$ and $\tau$ be $(p,q)$-clans.
Then $\gamma \leq \tau$ if and only if all three inequalities
\begin{enumerate}
\item $\gamma(i;+) \geq \tau(i;+)$,
\item $\gamma(i;-) \geq \tau(i; -)$, and
\item $\gamma(i,j) \leq \tau(i,j)$
\end{enumerate}
hold for all $1 \leq i<j \leq n$.
\end{theorem}

The unique maximum element of $\Clan_{p,q}$ in the inclusion order is $$\gamma_0:= 12\cdots q + \cdots + q \cdots 21.$$ (There are $p-q$ plus signs appearing in $\gamma_0$.) The $K$-orbit $\O_{\gamma_0}$ is open and dense in $\B$.

\begin{example}\label{ex.max.clan} We have
\[
\left(\gamma_0(i;+)\right)_{i=1}^n = (\underbrace{0,\ldots, 0}_{\textup{$q$ times}},1,2,\ldots, p),
\]
\[
\left(\gamma_0(i;-)\right)_{i=1}^n = (\underbrace{0,\ldots, 0}_{\textup{$p$ times}},1,2,\ldots, q),
\]
and
\begin{eqnarray*}\label{eqn.gamma0value1}
\gamma_0(i, j)=\left\{\begin{array}{ll}i & \textup { if } i \in [q], j \in [p], \\ 
q & \textup { if } i, j \in [p],\\
p+q-j & \textup{ if } i \in\{q+1, \ldots n\},  j \in\{p+1, \ldots, n\}. \\
\end{array} \quad\right. 
\end{eqnarray*}
Finally, if $i\in [q]$ and $j\in \{p+1, \ldots, n\}$ then
\begin{eqnarray*} \label{eqn.gamma0value2}
\gamma_0(i,j) = \min \{n-j, i\}.
\end{eqnarray*}
\end{example}

The following statement, which we record here for use in the next section, follows directly from the definition of the statistic $\gamma(i,j)$.

\begin{lemma}\label{lemma.gamma.dif} Let $\gamma\in \Clan_{p,q}$.
For all $i>1$,  $\gamma(i, j)-\gamma(i-1, j) \in\{0,1\}$ with $\gamma(i, j)-\gamma(i-1, j)=1$ if and only if there exists $t>j$ such that $c_i=c_t$.
 \end{lemma}

\subsection{The weak order} \label{ss.weak-order}

We now recall a formula of Brion for the cohomology class of a $K$-orbit closure $\ov{\O_\gamma}$ from~\cite{Brion2001}.  While there is a version of Brion's result for orbits of arbitrary spherical subgroups, we state here the result for the special case of the spherical subgroup $K=GL_p(\C)\times GL_q(\C)$ in $GL_n(\C)$. 

First, we require some terminology. Let $\Delta$ denote the subset of simple roots in the root system of $\mathfrak{gl}_n(\C)$ specified by our choice of Borel subgroup $B$. In particular, we have
\[
\Delta = \{\epsilon_i-\epsilon_{i+1} \mid i\in [n-1]\},
\]
where $\epsilon_i: \mathfrak{gl}_n(\C)\to \C$ is defined by $\epsilon_i(\x)=\x_{i,i}$.
For each $\alpha_i:= \epsilon_i - \epsilon_{i+1}\in \Delta$, let $P_i$ be the minimal parabolic subgroup defined by $P_i:= B \sqcup B s_iB$. Consider the canonical projection map $\pi_i: G/B \to G/P_i$. For each $\gamma\in \Clan_{p,q}$, the pull-back $\ov{\pi_i^{-1}(\pi_i(\O_\gamma))}$ contains a unique dense $K$-orbit, which we denote by $s_i\cdot \O_\gamma$. 
Notice that there might be more than one simple transposition giving the same $K$-orbit. Although this is not essential for the definition of our weak order, it will be important for us to keep track of these different simple transpositions. The \textit{weak order} on the set of $K$-orbits is the transitive closure of the relation defined by 
\begin{align}\label{A:geometriccovers}
\O_\gamma \prec \O_\tau \Leftrightarrow \tau\neq \gamma \textup{ and } \O_\tau = s_i\cdot \O_\gamma \;\text{for some}\; i\in [n-1].
\end{align}
We also write $\gamma \preceq \tau$ to denote the weak order on the set $\Clan_{p,q}$. It is clear that $\gamma\leq \tau$ whenever $\gamma\preceq \tau$. The clan $\gamma_0$ is the unique maximal element of $\Clan_{p,q}$ with respect to both the weak order and inclusion order.

We form an (oriented) graph on the vertex set $\Clan_{p,q}$ with edges $\gamma \to \tau$
whenever \eqref{A:geometriccovers} holds for some $s_i$, $i\in [n-1]$. 
In this case, we label the edge as follows:  
\[
\gamma\xrightarrow{\;s_i\;} \tau. 
\]
As we mentioned before, there can be more than one simple transposition $s_i$ with $i\in [n-1]$ giving the same cover relation in (\ref{A:geometriccovers}). Hence, an edge of our directed graph may possess multiple labels. We will use these labels in Section~\ref{S:Wsets}.

Given a directed path 
\[
P: \gamma=\gamma_1 \xrightarrow{\;s_{i_1}\;} \gamma_2 \xrightarrow{\;s_{i_2}\;} \gamma_3 \cdots \xrightarrow{\;s_{i_\ell}\;} \gamma_{\ell+1} = \gamma_0 
\]
from $\gamma$ to $\gamma_0$ we define $w(P) := s_{i_1}s_{i_2}\cdots s_{i_\ell}\in \S_n$.

\begin{definition} \label{def.Wset} For each $\gamma\in \Clan_{p,q}$,  the \textit{$W$-set} of the $K$-orbit $\O_\gamma$ is 
\[
W(\gamma):= \{w(P) \mid \text{$P$ a labeled directed path from $\gamma$ to $\gamma_0$}\}\subseteq \S_n.
\]
\end{definition}

We can now state Brion's formula \cite[Theorem 6]{Brion2001}.

\begin{theorem}[Brion]\label{thm.Brion}
Let $\gamma\in \Clan_{p,q}$. The $K$-orbit closure $\ov{\O_\gamma}$ has rational singularities and admits a flat degeneration to the reduced subscheme 
\begin{align*}
\bigcup_{w\in W(\gamma) } \overline{Bw_0 w B/B} \subset \B.
\end{align*}
In particular, we have 
\begin{align}\label{eqn.class}
[\ov{\O_\gamma}] = \sum_{w\in W(\gamma) } [\overline{Bw_0 w B/B}]
\end{align}
in the integral cohomology ring of $\B$. 
\end{theorem}

\begin{remark}  
Let us denote by $\mathcal{B}(G/K)$ the set of all $B$-orbit closures in a spherical homogeneous space $G/K$, where $G$ is a complex connected reductive algebraic group, and $K$ is a spherical subgroup of $G$. 
(As usual, $T$, $B$, and $W$ stand for a maximal torus in $G$, a Borel subgroup containing $T$ in $G$, and the Weyl group of $G$, respectively.) 
For $Y\in \mathcal{B}(G/K)$, the $W$-set of $Y$ consists of $w\in W$ such that the natural quotient morphism $\pi_{Y,w} : \overline{BwB}\times_B Y\to G/K$ is surjective and generically finite. 
It turns out that, by~\cite[Lemma 5]{Brion2001}, this definition is equivalent to a generalization of our Definition~\ref{def.Wset} to the setup of spherical homogeneous spaces.

Let $d(Y,w)$ denote the degree of $\pi_{Y,w}$. 
It turns out that this number is always a power of 2,~\cite[Lemma 5 (iii)]{Brion2001}. The real geometric usefulness of this integer is explained by Brion in~\cite[Theorem 6]{Brion2001}. In particular, the cohomology class corresponding to $Y$ in $H^*(G/B,\Z)$ 
is given by 
\[
[Y] = \sum_{w\in W(Y)} d(Y,w)[\overline{Bw_0 wB/B}].
\]
In our special case, where $K=GL_p(\C)\times GL_q(\C)$, 
the work of Vust~\cite{Vust1990} implies that each of these degrees is equal to $1$, implying  our identity~(\ref{eqn.class}). It also implies the vanishing of all higher cohomology spaces for the restrictions of effective line bundles from $G/B$ to $Y$.
\end{remark}

We now recall a combinatorial description for the weak order on $\Clan_{p,q}$ used in the work of the first author, Joyce, and Wyser~\cite{CJW}.  This description is most easily stated in terms of charged matchings.  A \textit{matching} on $[n]$ is a finite graph on the vertex set $[n]$ such that each vertex is either isolated or adjacent to precisely one other vertex. A \textit{charged matching} is a matching with an assignment of a $+$ or $-$ charge to each isolated vertex.

The set of $(p,q)$-clans is in bijection with the set of all \textit{charged matchings} on $[n]$ having $p-q$ more $+$'s than $-$'s.  Explicitly, we obtain a matching from a clan $\gamma=c_1c_2\cdots c_n$ by connecting $i$ and $j$ by an arc whenever $c_i=c_j\in \Z_+$ and recording all signed entries as charges on isolated vertices. We identify the set of $(p,q)$-clans with charged matchings throughout, but particularly in Section~\ref{S:Wsets} below.

\begin{example} The matching associated to the (5,3)-clan $\gamma=+1+-2+21$ is as follows. 

\begin{figure}[htp]
\begin{center}
\scalebox{.8}{
\begin{tikzpicture}

\begin{scope}[xshift=0cm]
 \node (x1) at (-4,0) {$\bullet$}; 
 \node (x2) at (-3,0) {$\bullet$};
 \node (x3) at (-2,0) {$\bullet$};
 \node (x4) at (-1,0) {$\bullet$};
 \node (x5) at (0,0) {$\bullet$};
 \node (x6) at (1,0) {$\bullet$};
 \node (x7) at (2,0) {$\bullet$};
 \node (x8) at (3,0) {$\bullet$};
  
\node at (-4,-.5) {$1$}; 
\node at (-3,-.5) {$2$};
\node at (-2,-.5) {$3$};
\node at (-1,-.5) {$4$};
\node at (0,-.5) {$5$};
\node at (1,-.5) {$6$};
\node at (2,-.5) {$7$};
 \node at (3,-.5) {$8$};

\node at (-4,+.5) {$+$}; 

\node at (-2,+.5) {$+$};
\node at (-1,+.5) {$-$};

\node at (1,+.5) {$+$};

\draw [ultra thick] (x5) to[out=90 ,in=90] (x7) ;
\draw [ultra thick,-] (x2) to[out=60 ,in=120] (x8) ;
\end{scope}
 \end{tikzpicture}
}
\end{center}
\end{figure}
\end{example}

From~\cite[Section 2.5]{CJW}, we get that the weak order on clans is the transitive closure of the covering relations 
\[
\gamma \xrightarrow{\; s_i \;} \gamma'
\]
where we obtain $\gamma'$ from $\gamma$ according to one of the following moves on the corresponding charged matchings, each of which is illustrated in Figure~\ref{F:coverrelations} below.

\begin{itemize}
\item Types IA1 and IA2: Switch the endpoint of a strand with an adjacent sign so as to lengthen the strand.
\item Type IB: Create a crossing from two disjoint strands at consecutive vertices.
\item Types IC1 and IC2: Create a nested pair of strands by uncrossing the ends of two crossing strands at consecutive vertices. 
\item Type II: Replace a pair of consecutive, opposite charges by a strand of length 1.
\end{itemize}

An astute reader will note that~\cite{CJW} actually studies the \textit{opposite weak order} on $\Clan_{p,q}$ so our Figure~\ref{F:coverrelations} reverses the covering relations as presented in Figure 2.5 of that reference.

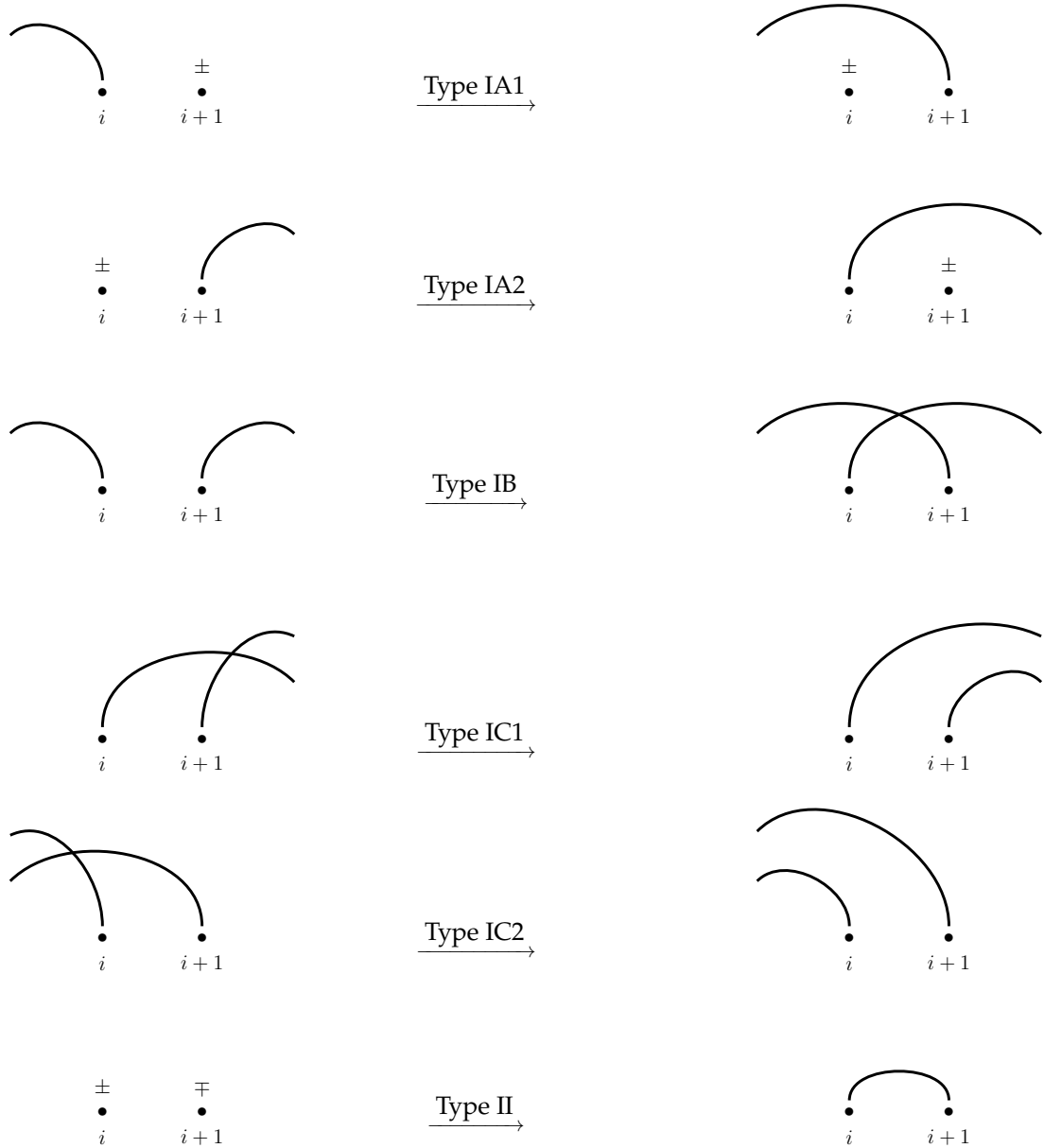
\begin{figure}[htp]
\begin{center}

\scalebox{.7}{
\begin{tikzpicture}

\begin{scope} [xshift=7.5cm, yshift=11.5cm]
\node (a) at (0,0) {$\bullet$}; 
\node (a1) at (-2,1) {{}};
\node (b) at (2,0) {$\bullet$}; 
\node at (0,.5) {$\pm$}; 
\node at (0,-.5) {$i$}; 
\node at (2,-.5) {${i+1}$}; 
\draw [ultra thick] (b) to[out=90 ,in=45] (a1) ;
\end{scope} 

\begin{scope} [xshift=-7.5cm, yshift=11.5cm]
\node (a) at (0,0) {$\bullet$}; 
\node (a1) at (-2,1) {{}};
\node (b) at (2,0) {$\bullet$}; 
\node at (2,.5) {$\pm$}; 
\node at (0,-.5) {$i$}; 
\node at (2,-.5) {${i+1}$}; 
\draw [ultra thick] (a) to[out=90 ,in=45] (a1) ;
\end{scope} 

\begin{scope} [xshift=0cm, yshift=11.5cm]
\node at (0,0) {$\xrightarrow{\;\text{\large{Type IA1}}\;}$};
\end{scope}

\begin{scope} [xshift=7.5cm, yshift=7.5cm]
\node (a) at (0,0) {$\bullet$}; 
\node (a1) at (4,1) {{}};
\node (b) at (2,0) {$\bullet$}; 
\node at (2,.5) {$\pm$}; 
\node at (0,-.5) {$i$}; 
\node at (2,-.5) {${i+1}$}; 
\draw [ultra thick] (a) to[out=90 ,in=135] (a1) ;
\end{scope} 

\begin{scope} [xshift=-7.5cm, yshift=7.5cm]
\node (a) at (0,0) {$\bullet$}; 
\node (a1) at (4,1) {{}};
\node (b) at (2,0) {$\bullet$}; 
\node at (0,.5) {$\pm$}; 
\node at (0,-.5) {$i$}; 
\node at (2,-.5) {${i+1}$}; 
\draw [ultra thick] (b) to[out=90 ,in=135] (a1) ;
\end{scope} 

\begin{scope} [xshift=0cm, yshift=7.5cm]
\node at (0,0) {$\xrightarrow{\;\text{\large{Type IA2}}\;}$};
\end{scope}

\begin{scope} [xshift=7.5cm, yshift=3.5cm]
\node (a) at (0,0) {$\bullet$}; 
\node (a1) at (4,1) {{}};
\node (b) at (2,0) {$\bullet$}; 
\node (bb) at (-2,1) {{}}; 
\node at (0,-.5) {$i$}; 
\node at (2,-.5) {${i+1}$}; 
\draw [ultra thick] (a) to[out=90 ,in=135] (a1) ;
\draw [ultra thick] (b) to[out=90 ,in=45] (bb) ;
\end{scope} 

\begin{scope} [xshift=-7.5cm, yshift=3.5cm]
\node (a) at (0,0) {$\bullet$}; 
\node (a1) at (4,1) {{}};
\node (b) at (2,0) {$\bullet$}; 
\node (bb) at (-2,1) {{}}; 
\node at (0,-.5) {$i$}; 
\node at (2,-.5) {${i+1}$}; 
\draw [ultra thick] (a) to[out=90 ,in=45] (bb) ;
\draw [ultra thick] (b) to[out=90 ,in=135] (a1) ;

\end{scope} 

\begin{scope} [xshift=0cm, yshift=3.5cm]
\node at (0,0) {$\xrightarrow{\;\text{\large{Type IB}}\;}$};
\end{scope}

\begin{scope} [xshift=7.5cm, yshift=-1.5cm]
\node (a) at (0,0) {$\bullet$}; 
\node (a1) at (4,2) {{}};
\node (b) at (2,0) {$\bullet$}; 
\node (bb) at (4,1) {{}}; 
\node at (0,-.5) {$i$}; 
\node at (2,-.5) {${i+1}$}; 
\draw [ultra thick] (a) to[out=90 ,in=155] (a1) ;
\draw [ultra thick] (b) to[out=90 ,in=135] (bb) ;
\end{scope} 

\begin{scope} [xshift=-7.5cm, yshift=-1.5cm]
\node (a) at (0,0) {$\bullet$}; 
\node (a1) at (4,1) {{}};
\node (b) at (2,0) {$\bullet$}; 
\node (bb) at (4,2) {{}}; 
\node at (0,-.5) {$i$}; 
\node at (2,-.5) {${i+1}$}; 
\draw [ultra thick] (a) to[out=90 ,in=135] (a1) ;
\draw [ultra thick] (b) to[out=90 ,in=155] (bb) ;

\end{scope} 

\begin{scope} [xshift=0cm, yshift=-1.5cm]
\node at (0,0) {$\xrightarrow{\;\text{\large{Type IC1}}\;}$};
\end{scope}

\begin{scope} [xshift=7.5cm, yshift=-5.5cm]
\node (a) at (0,0) {$\bullet$}; 
\node (a1) at (-2,1) {{}};
\node (b) at (2,0) {$\bullet$}; 
\node (bb) at (-2,2) {{}}; 
\node at (0,-.5) {$i$}; 
\node at (2,-.5) {${i+1}$}; 
\draw [ultra thick] (a) to[out=90 ,in=45] (a1) ;
\draw [ultra thick] (b) to[out=90 ,in=45] (bb) ;
\end{scope} 

\begin{scope} [xshift=-7.5cm, yshift=-5.5cm]
\node (a) at (0,0) {$\bullet$}; 
\node (a1) at (-2,2) {{}};
\node (b) at (2,0) {$\bullet$}; 
\node (bb) at (-2,1) {{}}; 
\node at (0,-.5) {$i$}; 
\node at (2,-.5) {${i+1}$}; 
\draw [ultra thick] (a) to[out=90 ,in=25] (a1) ;
\draw [ultra thick] (b) to[out=90 ,in=45] (bb) ;

\end{scope} 

\begin{scope} [xshift=0cm, yshift=-5.5cm]
\node at (0,0) {$\xrightarrow{\;\text{\large{Type IC2}}\;}$};
\end{scope}

\begin{scope} [xshift=7.5cm, yshift=-9cm]
\node (a) at (0,0) {$\bullet$}; 
\node (b) at (2,0) {$\bullet$}; 
\node at (0,-.5) {$i$}; 
\node at (2,-.5) {${i+1}$}; 
\draw [ultra thick] (a) to[out=90 ,in=90] (b) ;
\end{scope} 

\begin{scope} [xshift=-7.5cm, yshift=-9cm]
\node (a) at (0,0) {$\bullet$}; 
\node (b) at (2,0) {$\bullet$}; 
\node at (0,-.5) {$i$}; 
\node at (2,-.5) {${i+1}$}; 
\node at (0,.5) {$\pm$}; 
\node at (2,.5) {$\mp$}; 
\end{scope} 

\begin{scope} [xshift=0cm, yshift=-9cm]
\node at (0,0) {$\xrightarrow{\;\text{\large{Type II}}\;}$};
\end{scope}

 \end{tikzpicture}
}
\end{center}
\caption{Cover relations of the weak order on $\mathbf{Clan}_{p,q}$.}
\label{F:coverrelations}
\end{figure}

\section{Irreducible Hessenberg varieties $\Hess(\x_{p,q},\m)$} \label{S:Irreducible}

In this section, we classify all irreducible Hessenberg varieties of the form $\Hess(\x_{p,q},\m)$ and prove Theorem~\ref{thm1.intro} from the Introduction. 
To begin, we identify the $K$-orbits that are contained in $\Hess(\x_{p,q},\m)$.  

\begin{proposition} \label{prop.orbitsinhess}
The $K$-orbit $\O_\gamma$ associated to the $(p,q)$-clan $\gamma=c_1c_2\cdots c_n$ lies in $\Hess(\x_{p,q},\m)$ if and only if $m_i \geq j$ whenever $c_i=c_j \in \Z_+$ with $i<j$.
\end{proposition}

\begin{proof} It suffices to determine which clans $\gamma$ satisfy $V(\gamma)_\bullet \in \Hess(\x_{p,q},\m)$, where $V(\gamma)$ is the flag representative of $\O_\gamma$ specified in Definition~\ref{def.flag.rep} above.
We observe first that each $e_i$ is an eigenvector for $\x_{p,q}$ and that if $v_i \in \{e_r+e_s,e_r-e_s\}$ with $r \in [p]$ and $p<s \leq n$, then $\C\{v_i,\x_{p,q} v_i\}=\C\{e_r,e_s\}$.  Thus $V(\gamma)_i+\x_{p,q} V(\gamma)_i$ is spanned by those standard basis vectors $e_k$ such that one of
\begin{itemize}
\item there is some $j \in [i]$ with $v_j=e_k$, or
\item there is some $a \in [i]$ with $v_a=e_r+e_s$ and $k \in \{r,s\}$
\end{itemize}
holds.  On the other hand, the standard basis vector $e_k$ is an element of $V(\gamma)_{m_i}$ if and only if one of
\begin{itemize}
\item there is some $j \in [m_i]$ with $v_j=e_k$, or
\item there is some $b \in [m_i]$ with $v_b=e_r-e_s$ and $k \in \{r,s\}$
\end{itemize}
holds.  Indeed, if there is no $j \in [m_i]$ with $v_j=e_k$, then $e_k \in V(\gamma)_{m_i}$ if and only if there are $a,b \in [m_i]$ with $v_a=e_r+e_s$, $v_b=e_r-e_s$ and $k \in \{r,s\}$.  In this case, $a<b$ and $c_a=c_b$ by definition of the flag $V(\gamma)_\bullet$.  The proposition follows.
\end{proof} 

As any $\Hess(\x_{p,q},\m)$ is a union of $K$-orbits, the next proposition follows immediately from the definitions.

\begin{proposition} \label{prop.irr.criterion.1}
The Hessenberg variety $\Hess(\x_{p,q},\m)$ is irreducible if and only if, among the clans corresponding to $K$-oribts contained in $\Hess(\x_{p,q}, \m)$, there is a unique maximal one with respect to the inclusion order.
\end{proposition}

Let us specify two particular clans $\sigma$ and $\tau$ in $\Clan_{p,q}$ by
\[
\sigma:= \underbrace{++ \cdots +}_{\text{$p$ times}}\underbrace{--\cdots -}_{\text{$q$ times}}  
\]
and
\[
\tau:= \underbrace{-- \cdots -}_{\text{$q$ times}}\underbrace{++\cdots +}_{\text{$p$ times}}.  
\]
Observe that $\sigma(i;-)=0$ for $i \leq p$ and $\tau(i;+)=0$ for $i \leq q$.  
Under our assumption that $p \geq q$, the next claim follows.

\begin{lemma} \label{lemma.threeconditions}
If $\gamma=c_1c_2\cdots c_n\in \Clan_{p,q}$ such that $\sigma \leq \gamma$ and $\tau \leq \gamma$, then all of
\begin{enumerate}[label=(\alph*)]
\item $c_i=i$ for each $i \in [q]$,
\item $c_i=+$ for $q<i \leq p$, and
\item $\{c_i \mid p<i \leq p+q\}=[q]$
\end{enumerate}
hold.
In particular, if $\ov{\O_\gamma}=\Hess(\x_{p,q},\m)$ for some Hessenberg vector $\m$, then $\gamma$ satisfies all three conditions.
\end{lemma}

\begin{proof}
As each of $\O_\sigma$ and $\O_\tau$ lies in $\ov{\O_\gamma}$, we have
\begin{itemize}
\item $\gamma(i;-) \leq \sigma(i;-)=0$ for $i \leq p$ and
\item $\gamma(i;+) \leq \tau(i;+)=0$ for $i \leq q$
\end{itemize}
by Theorem~\ref{thm.inclusion}. These conditions imply that $\gamma$ cannot contain any signs or pairs of positive integers within the first $q$ entries and cannot contain any minus signs or pairs of positive integers within the first $p$ entries. Since $\gamma$ contains at most $q$ natural number pairs, conditions (a) and (b) now follow.  Condition (c) follows from (a) and (b) and the fact that $\gamma$ is a clan. The last statement of the lemma follows immediately, as both $\O_\sigma$ and $\O_\tau$ lie in $\Hess(\x_{p,q},\m)$ for every Hessenberg vector $\m$ by Proposition~\ref{prop.orbitsinhess}.
\end{proof}

We can rewrite condition (c) from Lemma \ref{lemma.threeconditions} as
\textit{\begin{itemize}
\item[(c')] there is some $w \in \S_q$ such that $c_{p+i}=w(i)$ for each $i \in [q]$.
\end{itemize}}
Given a clan $\gamma$ satisfying (a),(b), and (c'), we write $\gamma_w$ for $\gamma$.  Thus, for each $w\in \S_q$ we obtain a unique clan $\gamma_w=c_1^wc_2^w \cdots c_n^w$ defined by
\begin{itemize}
\item $c_i^w= +$ for all $q<i \leq p$, and
\item $c_i^w=c_{p+w^{-1}(i)}^w=i$ for all $i\in [q]$.
\end{itemize}
Note that $\gamma_0=\gamma_{y_0}$ where $y_0$ is the longest permutation in $\S_q$. In fact, the collection of all such $(p,q)$-clans forms is precisely the inclusion-interval $[\gamma_e, \gamma_0]$ in $\Clan_{p,q}$.

\begin{lemma} \label{lemma.w-statistics} Let $\gamma_e = 12\cdots q + \cdots +12\cdots q$ be the clan corresponding to the identity in  $\S_q$.  If $\gamma_e \leq \gamma$, then $\gamma=\gamma_w$ for some $w\in \S_q$.  In particular, $\gamma_w(i;+)=\gamma_0(i;+)$ and $\gamma_w(i;-)=\gamma_0(i;-)$ for all $i$, and $\gamma_w(i,j) = \gamma_0(i,j)$ whenever $i, j \in [p]$ or $i, j \in\{p+1, \ldots, n\}$.
\end{lemma}
\begin{proof} We observe that both $\sigma\leq \gamma_e$ and $\tau \leq \gamma_e$. Thus $\gamma$ satisfies each of the conditions (a), (b), and (c') by Lemma~\ref{lemma.threeconditions}, and the first assertion of the lemma is proved.  To prove the second, we observe that the equality of the various statistics holds in the case of $w=e$; cf.~Example~\ref{ex.max.clan}. The general case now follows since $\gamma_0(i;-)\leq \gamma_w(i;-)\leq \gamma_e(i;-)$,  $\gamma_0(i;+)\leq \gamma_w(i;-)\leq \gamma_e(i;+)$, and $\gamma_e(i,j)\leq \gamma_w(i,j)\leq \gamma_0(i,j)$ by Theorem~\ref{thm.inclusion} in all cases. 
\end{proof}

The inclusion order on the clans in the interval $[\gamma_e, \gamma_0]$ is greatly simplified. Indeed, the only case in which the statistics appearing in Theorem~\ref{thm.inclusion} can differ is when considering $\gamma_w(i,j)$ with $i\in [q]$ and $j\in \{p+1,p+2,\ldots, n\}$. In that situation, we obtain the following.

\begin{lemma}\label{lemma.ij-statistic} For all $w\in \S_q$ and all $i \in [q], j \in\{p+1, \ldots, n\}$,
$$\begin{aligned}
\gamma_w(i, j) & =\left|\left\{w^{-1}(1), \ldots, w^{-1}(i)\right\} \cap\{j-p+1, \ldots, q\}\right| \\
& =\left|\left\{k \leq i \mid w^{-1}(k)>j-p\right\}\right| .
\end{aligned}
$$
\end{lemma}
\begin{proof} We have a pair $s<t$ such that $s \leq i<j<t$ and $c_s^w=c_t^w$ if and only if $t=w^{-1}(s)+p$ by definition of $\gamma_w$.
\end{proof}

\begin{example}
Consider $(5,3)$-clans $\gamma_w = 123++213$ and $\gamma_{w'} = 123++132$.
For the clan $\gamma_w$, $w(1)=2, w(2)=1,$ and $w(3)=3$, while for $\gamma_{w'}$ we have $w'(1)=1, w'(2)=3,$ and $w'(3)=2$. 
The definition of $\gamma(i;j)$ and straightforward calculation give us that
$$
(\gamma_w(i,j))_{i,j=1}^n=\left( \begin{array}{cccccccc} 
0 & 1 & 1 & 1 & 1 & 1 & 0 & 0 \\ 
0 & 0 & 2 & 2 & 2 & 1 & 0 & 0 \\ 
0 & 0 & 0 & 3 & 3 & 2 & 1 & 0 \\ 
0 & 0 & 0 & 0 & 3 & 2 & 1 & 0 \\ 
0 & 0 & 0 & 0 & 0 & 2 & 1 & 0 \\ 
0 & 0 & 0 & 0 & 0 & 0 & 1 & 0 \\ 
0 & 0 & 0 & 0 & 0 & 0 & 0 & 0 \\ 
0 & 0 & 0 & 0 & 0 & 0 & 0 & 0 \end{array} \right)
\quad
(\gamma_{w'}(i,j))_{i,j=1}^n=\left( \begin{array}{cccccccc} 
0 & 1 & 1 & 1 & 1 & 0 & 0 & 0 \\ 
0 & 0 & 2 & 2 & 2 & 1 & 1 & 0 \\ 
0 & 0 & 0 & 3 & 3 & 2 & 1 & 0 \\ 
0 & 0 & 0 & 0 & 3 & 2 & 1 & 0 \\ 
0 & 0 & 0 & 0 & 0 & 2 & 1 & 0 \\ 
0 & 0 & 0 & 0 & 0 & 0 & 1 & 0 \\ 
0 & 0 & 0 & 0 & 0 & 0 & 0 & 0 \\ 
0 & 0 & 0 & 0 & 0 & 0 & 0 & 0 \end{array} \right).
$$ 
Moreover, it also follows from Lemma~\ref{lemma.ij-statistic} that
\[\begin{array}{ccc}
 \gamma_w(1,6)=1=|\{1\}|  & \gamma_w(1,7)=0  &  \gamma_w(1,8)=0 \\
\gamma_w(2,6)=1=|\{1\}|  & \gamma_w(2,7)=0  &  \gamma_w(2,8)=0 \\
\gamma_w(3,6)=2=|\{1, 3\}|  & \gamma_w(3,7)=1=|\{3\}| &  \gamma_w(3,8)=0.
\end{array}\]
and, 
\[\begin{array}{ccc}
 \gamma_{w'}(1,6)=0  & \gamma_{w'}(1,7)=0  &  \gamma_{w'}(1,8)=0 \\
\gamma_{w'}(2,6)=1=|\{2\}|  & \gamma_{w'}(2,7)=1=|\{2\}|  &  \gamma_{w'}(2,8)=0 \\
\gamma_{w'}(3,6)=2=|\{2, 3\}|   & \gamma_{w'}(3,7)=1=|\{2\}| &  \gamma_{w'}(3,8)=0.
\end{array}\]
Therefore, the statistics from Theorem~\ref{thm.inclusion} only differ for these clans when $(i;j)$ is either $(1,6)$ or $(2,7)$.

\end{example}
 
Our work above tells us that if $\Hess(\x_{p,q},\m)$ is the closure of a single $K$-orbit, then it is equal to $\O_{\gamma_w}$ for some $w\in \S_q$.

\begin{corollary} \label{cor.xpqorbitclosure}
Let $\m=(m_1,\ldots,m_n)$ be a Hessenberg vector.  If there is some clan $\gamma$ such that $\Hess(\x_{p,q},\m)=\ov{\O_\gamma}$ then $\gamma=\gamma_w$ for some $w \in \S_q$.  Furthermore, $(m_1-p,m_2-p,\ldots,m_q-p)$ is a Hessenberg vector of length $q$ and $m_i=n$ for all $i\geq q$.
\end{corollary}

\begin{proof}
It follows immediately from Lemma~\ref{lemma.threeconditions} that $\gamma=\gamma_w$ for some $w \in \S_q$. 
By Proposition \ref{prop.orbitsinhess}, for each $i \in [q]$ we have $m_i \geq p+w^{-1}(i)$.  Thus, for each $i\in [q]$,
\[
p+w^{-1}(i)\leq m_i \leq p+q \Leftrightarrow w^{-1}(i) \leq m_i -p \leq q. 
\]
It follows that $(m_1-p,m_2-p,\ldots,m_q-p)$ is a sequence of positive integers satisfying $m_i-p \leq q$; it is also weakly increasing since $\m$ is. 
Moreover, for each $i \in [q]$, 
$$
m_i-p \geq \max\{w^{-1}(j) \mid j \leq i\} \geq i.
$$
This concludes the proof.
\end{proof}

It follows from Corollary \ref{cor.xpqorbitclosure} that there are at most $\cat_q=\frac{1}{q+1}{{2q} \choose {q}}$ Hessenberg vectors $\m$ of length $p+q$ such that $\Hess(\x_{p,q},\m)$ is a $K$-orbit closure as there are $\cat_q$ Hessenberg vectors of length $q$.  
We aim to show that there are exactly $\cat_q$ such $\m$, and classify the set of $\cat_q$ clans $\gamma$ such that $\Hess(\x_{p,q},\m)=\ov{\O_\gamma}$.

\begin{lemma} \label{prop.312free} 
Assume $u \in \S_q$ and $\m$ is a Hessenberg vector such that $\O_{\gamma_u} \subseteq \Hess(\x_{p,q},\m)$.  If there exist $i<j<k$ such that $u^{-1}(i)>u^{-1}(k)>u^{-1}(j)$ then there is some $w \in \S_q$ such that $\gamma_u \leq \gamma_w$ and $\O_{\gamma_w} \subseteq \Hess(\x_{p,q},\m)$.
\end{lemma}

\begin{proof} 
Let $w$ be obtained from $u$ by switching $j$ and $k$.  Direct examination shows that for all $a,b \in [n]$, all of $\gamma_w(a;+) \leq \gamma_u(a;+)$, $\gamma_w(a;-) \leq \gamma_u(a;-)$ and $\gamma_w(a,b) \geq \gamma_u(a,b)$ hold.  Thus $\gamma_u \leq \gamma_w$ by Theorem~\ref{thm.inclusion}.  Assume for contradiction that there exists $a,b \in [n]$ with $a<b$ and $s \in [q]$ such that $c^w_a=c^w_b=s$ and $b>m_a$.  By definition of the clan $\gamma_w$, $a \in [q]$ and $b=p+w^{-1}(a)$.  Since $\O_{\gamma_u} \subseteq \Hess(\x_{p,q},\m)$, it must be that $a \in \{j,k\}$ and so $b \in \{p+w(j),p+w(k)\}$.  However,
$$
m_a>m_i \geq p+u^{-1}(i)=p+w^{-1}(i)>p+\max\{w^{-1}(j),w^{-1}(k)\} \geq b,
$$ 
giving the desired contradiction.
\end{proof}

\begin{corollary} \label{cor.312free}
Let $w \in \S_q$.  If there is some Hessenberg vector $\m$ such that $\Hess(\x_{p,q},\m)=\ov{\O_{\gamma_w}}$, then $w^{-1}$ avoids the pattern $312$, hence $w$ avoids $231$.
\end{corollary}

There are $\cat_q$ elements $w \in \S_q$ avoiding $231$.  Let 
\[
\Clan_{p,q}^{231}:= \{ \gamma_w\mid w\in \S_q,\, \textup{$w$ avoids $231$}  \}.
\]
We prove below that this set of clans parameterize irreducible semisimple Hessenberg varieties of the form $\Hess(\x_{p,q}, \m)$. For each $w\in \S_q$ avoiding the pattern $231$, define a length $n=p+q$ Hessenberg vector $\m(w)$ by
$$
m(w)_i:=\begin{cases}\max \left\{w^{-1}(k)+p \mid k \leq i\right\} & i \leq q, \\ n & i>q .\end{cases}
$$
We can now state the main theorem of this section.

\begin{theorem} \label{thm.irreducible} For each $w\in \S_q$ avoiding the pattern $231$, the Hessenberg variety $\Hess(\x_{p,q},\m(w))$ is irreducible and equal to the closure of the $K$-orbit $\O_{\gamma_w}$.  Furthermore, every irreducible Hessenberg variety defined using the semisimple matrix $\x_{p,q}$ is of this form. 
\end{theorem}

Our proof of~Theorem \ref{thm.irreducible} requires the following technical lemma.

\begin{lemma}\label{lemma.technical} Let $w\in \S_q$ be $231$-free and $\m=\m(w)$ the associated Hessenberg vector defined above.  Let $i \in [q]$ such that $w^{-1}(i)+p< m_i$. Then for all $j$ with $w^{-1}(i)+p \leq j< m_i$, $\gamma_w(i-1, j)= m_i-j$.
\end{lemma}
\begin{proof} By Lemma~\ref{lemma.ij-statistic}, $\gamma_w(i-1, j)=\left|\left\{k<i \mid w^{-1}(k)>j-p\right\}\right|$.
By definition of the Hessenberg vector $\m$, there exists $a<i$ such that $m_i=w^{-1}(a)+p$.
Our assumptions imply that
$$
j<m_i=w^{-1}(a)+p \Rightarrow j-p<w^{-1}(a)\quad \text{and} \quad w^{-1}(i)<w^{-1}(a).
$$
Let $k \in\{1,2, \ldots, q\}$ such that $j-p<w^{-1}(k) \leq w^{-1}(a)$.
Note that $j-p<w^{-1}(k)$ implies $w^{-1}(i)<w^{-1}(k)$ and $k \neq i$.

If $k>i$ then $w^{-1}(k)<w^{-1}(a) $ and $w^{-1}$ contains $w^{-1}(a) w^{-1}(i) w^{-1}(k)$ as a subsequence, contradicting the fact that $w^{-1}$ is $312$-free. Thus, we must have $k<i$. This shows that
$$
\left\{k \in[q] \mid {j-p}<w^{-1}(k) \leq w^{-1}(a)\right\} \subseteq\left\{k<i \mid w^{-1}(k)>j-p\right\}.
$$
The sets are actually equal, since $$w^{-1}(a)=\max \left\{w^{-1}(1), \ldots, w^{-1}(i)\right\}.$$ 
We conclude $\gamma_w(i-1, j)=w^{-1}(a)-(j-p)=m_i-j$, as desired. 
\end{proof}

\begin{proof}[Proof of Theorem~\ref{thm.irreducible}] By Proposition~\ref{prop.irr.criterion.1}, Corollary~\ref{cor.xpqorbitclosure}, and Corollary~\ref{cor.312free}, every irreducible Hessenberg variety defined using $\x_{p,q}$ is equal to $\overline{\O_{\gamma_w}}$ for some $\gamma_w\in \Clan_{p,q}^{231}$.
To complete the proof we show that $\Hess(\x_{p,q}, \m(w)) = \ov{\O_{\gamma_w}}$. It follows immediately from the definition of the Hessenberg vector $\m(w)$ and Proposition~\ref{prop.orbitsinhess} that $\O_{\gamma_w} \subset \Hess(\x_{p,q},\m(w))$.  It therefore suffices to show that if $\O_\gamma\subset \Hess(\x_{p,q},\m(w))$ for some $\gamma=c_1c_2\cdots c_n\in \Clan_{p,q}$, then $\gamma \leq \gamma_w$.

By Theorem~\ref{thm.inclusion} and Lemma~\ref{lemma.w-statistics}, we must prove $\gamma(i, j) \leq \gamma_w(i, j)$ for all $i \in [q]$ and $j \in\{p+1, \ldots, n\}$.
Seeking a contradiction, suppose $\gamma(i, j)>\gamma_w(i, j)$. We may assume $i$ is minimal with respect to this property. We write $\m=\m(w)$ throughout, to simplify notation.

Consider first the case $i=1$. Note that $\gamma(1, j), \gamma_w(1, j) \in\{0,1\}$ so we must have $\gamma(1, j)=1$ and $\gamma_w(1, j)=0$. The latter implies $w^{-1}(1)+p \leq j$ by Lemma~\ref{lemma.ij-statistic}. On the other hand, $\gamma(1, j)=1$ implies by Lemma~\ref{lemma.gamma.dif} that there exists $t>j$ such that $c_1=c_t$. Now $m_1=w^{-1}(1)+p<t$, contradicting the assumption that $\O_\gamma \subset \Hess(\x_{p, q}, \m)$.

Now assume $i>1$. We have both $\gamma(i-1, j) \leq \gamma_w(i-1, j)$ and $\gamma(i, j)>\gamma_w(i, j)$. By Lemma~\ref{lemma.gamma.dif} this can only be the case if
\begin{eqnarray}\label{eqn1}
\gamma(i, j)-\gamma(i-1, j)=1,
\end{eqnarray}
and
\begin{eqnarray}\label{eqn2}
\gamma_w(i, j)-\gamma_w(i-1, j)=0.  
\end{eqnarray}
We may furthermore conclude that
\begin{eqnarray}\label{eqn3}
\gamma(i-1, j)=\gamma_w(i-1, j)
\end{eqnarray}
since otherwise, $\gamma(i-1, j)<\gamma_w(i-1, j)$ and
$$
\gamma(i, j)=\gamma(i-1, j)+1 \leq \gamma_w(i-1, j)=\gamma_w(i, j)_{,}
$$
contradicting our assumption that $\gamma(i, j)>\gamma_w(i, j)$.

By Lemma~\ref{lemma.gamma.dif}, Equation~\eqref{eqn1} implies that there exists $t>j$ such that $c_i=c_t$.
From equation~\eqref{eqn2}, we get that
\begin{eqnarray*}
&&\left|\left\{w^{-1}(1), \ldots, w^{-1}(i)\right\} \cap\{j-p+1, \ldots, q\}\right| = \\&& \quad \quad \quad\quad\quad\quad\quad \left|\left\{w^{-1}(1), \ldots, w^{-1}(i-1)\right\} \cap\{j-p+1, \ldots, q\}\right|
\end{eqnarray*}
so $w^{-1}(i) \leq j-p$, implying $w^{-1}(i)+p \leq j$.

If $m_i=w^{-1}(i)+p$ then we have $m_i<t$, a contradiction to $\O_\gamma \subset \Hess(\x_{p, q}, \m)$. We obtain the same contradiction if $j \geq m_i$ so we may now assume both $m_i>w^{-1}(i)+p$ and $j<m_i$. By Lemma~\ref{lemma.technical} and Equation~\eqref{eqn3}, $\gamma(i-1, j)=m_i-j$.
This implies there are precisely $m_i-j$ pairs $(a<b)$ such that $a \leq i-1<j<b$ and $c_a=c_b$.
As $\O_{\gamma} \subseteq \Hess(\x_{p, q}, \m)$ we have $b \leq m_a \leq m_i$ in each case.
There are only $m_i-j$ values $b$ such that $j<b \leq m_i$, and each such position in the clan $\gamma$ is occupied by $c_b$ for one of the pairs counted by $\gamma(i-1, j)$. This forces $t>m_i$, another contradiction. We conclude $\Hess(\x_{p,q},\m)=\ov{\O_{\gamma_w}}$, as desired.
\end{proof}

The dimension of the $K$-orbit $\O_\gamma$ associated to $\gamma=c_1c_2\cdots c_n\in \Clan_{p,q}$ is 
$$
\dim \O_\gamma = \ell(\gamma)+\frac{p(p-1)}{2}+\frac{q(q-1)}{2},
$$ 
where
$$
\ell(\gamma): = \sum_{\substack{c_i=c_j\in \N \\ i<j}} \left(\, j-i-|\{ k\in \N \mid c_s=c_t=k \textup{ for } s<i<t<j \}|\, \right)
$$
by~\cite{Yamamoto}. We apply this formula to compute the dimension of the irreducible Hessenberg variety $\Hess(\x_{p,q}, \m(w))$. We first require the following technical lemma.

\begin{lemma} \label{lem.lengthvshessenbergvector}
Given $w \in \S_n$, define a sequence ${\mathbf h}={\mathbf h}(w):=(h_1,\ldots,h_n)$ by
$$
h_i:=\max\{w(k) \mid k \leq i\}.
$$
If $w$ is $312$-free, then
$$
\ell(w)=\sum_{i=1}^n (h_i-i).
$$
\end{lemma}

\begin{proof}
Write $w=w(1)\ldots w(n)$ in one-line notation and find $k$ such that $w(k)=n$.  If $n>1$ let $w^\prime$ be obtained from $w$ by erasing $n$ from the given one-line representation and let ${\mathbf h}^\prime$ be obtained from $w^\prime$ as ${\mathbf h}$ was obtained from $w$.   
For $i \in [n]$, set
$$
\invv_i(w):=\mid \{j>i\mid w(j)<w(i)\} \mid,
$$ 
and define $\invv_i(w^\prime)$ similarly.  

We will show by induction on $n$ that $\invv_i(w)=h_i-i$ for every $i$, from which the lemma follows. The case $n=1$ is trivial. Assume $n>1$. We observe that if $i \geq k$ then $w_j<w_i$ for all $j>i$ (since $w$ is $312$-free) hence $\invv_i(w)=n-i=h_i-i$.  If $i<k$ then
$$
\invv_i(w)=\invv_i(w^\prime)=h^\prime_i-i=h_i-i,
$$
the second equality following from the inductive hypothesis.
\end{proof}

Recall that $\pi_{\m(w)}$ denotes the Dyck path associated with the Hessenberg vector $\m(w)$ as in Section~\ref{S:Notation}. Our work above shows that $\dim \Hess(\x_{pq}, \m(w))$ is given by the area of $\pi_{\m(w)}$.

\begin{corollary} \label{cor.dimension} For each $w\in \S_q$ avoiding the pattern $231$, the Hessenberg variety $\Hess(\x_{p,q}, \m(w))$ is irreducible of dimension
\begin{eqnarray}\label{eqn.dimension}
\dim \Hess(\x_{p,q},\m(w)) = \ell(w) + pq+ \frac{p(p-1)}{2} = \area(\pi_{\m(w)}).
\end{eqnarray}
\end{corollary}
\begin{proof}  
Recall $c_i^w=c_j^w\in \N$ if and only if $j=w^{-1}(i)+p$. Keeping also in mind that $\ell(w)=\ell(w^{-1})$, we therefore have 
\[
|\{k\in \N \mid c_s=c_t=k \textup{ for } s<i<t<w^{-1}(i)+p \}| = |\{ s<i \mid w^{-1}(s)<w^{-1}(i) \}|,
\]
and thus 
\begin{eqnarray*}
\sum_{i\in [q]} |\{k\in \N \mid c_s=c_t=k \textup{ for } s<i<t<j \}|  = \ell(w_0) - \ell(w^{-1})=\frac{q(q-1)}{2}-\ell(w).
\end{eqnarray*}
We now obtain
\begin{eqnarray*}
\ell(\gamma_w) &=& \sum_{i\in [q]} \left( w^{-1}(i)-i +p \right) - \sum_{i\in [q]} |\{k\in \N \mid c_s=c_t=k \textup{ for } s<i<t<j \}|\\
&=& pq - \left( \frac{q(q-1)}{2}-\ell(w)\right) = \ell(w)+ pq -  \frac{q(q-1)}{2}.
\end{eqnarray*}
As $\dim \Hess(\x_{p,q},\m(w)) = \dim \O_{\gamma_w}$ by Theorem~\ref{thm.irreducible}, this proves the first equality in~\eqref{eqn.dimension}. 

To prove the second we observe first that if $\m(w)=(m_1,\ldots,m_n)$ then
\begin{eqnarray*}
\area(\pi_{m(w)}) & = & \sum_{i=1}^n (m_i-i) \\ & = & pq+\sum_{i=1}^q \max\{w^{-1}(k) \mid k \leq i\}- \sum_{i=1}^q i+pn-\sum_{i=q+1}^n i \\ & = & p^2+2pq-{{p+q+1} \choose {2}}+{{q+1} \choose {2}}+\sum_{i=1}^q \max\{w^{-1}(k) \mid k \leq i\}-\sum_{i=1}^q i \\ & = & pq+\frac{p(p-1)}{2}+\sum_{i=1}^q \max\{w^{-1}(k) \mid k \leq i\}- \sum_{i=1}^q i.
\end{eqnarray*}
We complete the proof by applying Lemma \ref{lem.lengthvshessenbergvector} to~$w^{-1}$.
\end{proof}

\begin{remark}
It follows from Corollary~\ref{cor.dimension} and the seminal work~\cite{demariprocesishayman} of De Mari, Procesi, and Shayman on Hessenberg varieties that if $\Hess(\x_{p,q},\m)$ is irreducible and ${\mathsf s}$ is an $n \times n$ regular semisimple matrix, then
$$
\dim\Hess(\x_{p,q},\m)=\dim\Hess({\mathsf s},\m).
$$
Indeed, in the case of a regular semisimple element ${\mathsf s}$, it is easy to see from~\cite[Theorem 6]{demariprocesishayman} that the dimension of $\Hess({\mathsf s},\m)$ is precisely the area of $\pi_\m$.
\end{remark}


\section{W-sets and cohomology classes}\label{S:Wsets}

We now turn our attention to computing the $W$-sets introduced in Section~\ref{ss.weak-order} above for the clans $\gamma_w$ with $w\in \S_q$.  Our work below shows that the restriction of the weak order to the interval $[\gamma_e, \gamma_0]$ in $\Clan_{p,q}$ can be identified with the two-sided weak order on $\S_q$ (see Theorem~\ref{thm.weak-order-Sq} below).  As a result, we give a concrete formula for the class $[\overline{\O}_{\gamma_w}]$ and, in particular, the class of any Hessenberg variety $\Hess(\x_{p,q}, \m(w))$.  Finally, as an application of our results, we prove that the product of $[\overline{\O}_{\gamma_w}]$ with any Schubert divisor is a multiplicity-free sum of Schubert polynomials.   

Recall that \textit{left weak order} $\leq_L$ on the symmetric group $\S_q$ is the partial order defined by the covering relations 
\[
w <_L s_iw  \textup{ where } i\in [n-1] \textup{ is such that } w^{-1}(i)<w^{-1}(i+1).
\]
The left multiplication by $s_i$ interchanges the order of $i$ and $i+1$ in the one-line notation for $w$.  For example, $51324 <_L 52314$.
Similarly, the \textit{right weak order} $\leq_R$ on $\S_q$ is the partial order defined by the covering relations
\[
w <_R ws_i  \textup{ where } i\in [n-1] \textup{ is such that }  w(i)<w(i+1).
\]
The right multiplication by $s_i$ interchanges the entries in positions $i$ and $i+1$ of the one-line notation for $w$. For example, $51324 <_R 53124$. 

We call the partial order $\preceq$ on $\S_q$ that is generated by the covering relations of both of the left and the right weak orders the \textit{two-sided weak order} on $\S_q$.

\begin{example}
    In Figure~\ref{F:S4}, we depict the two-sided weak order on $\S_4$.
    The blue (double) edges correspond to the cover relations that are admitted by both of the orders $\leq_L$ and $\leq_R$. The ordinary edges correspond to a covering relation of either $\leq_R$ or $\leq_L$, but not both.  Our figure shows that the two-sided weak order on $\S_4$ is not isomorphic to the Bruhat (i.e., inclusion) order; for example, $s_1s_2s_1 = 3214 \leq s_2s_1s_3s_2 = 3412$ in Bruhat order but Figure~\ref{F:S4} shows that $3214$ is not below $3412$ in the two-sided weak order.
\begin{figure}[htp]
\begin{center}

\scalebox{.65}{
\begin{tikzpicture}[scale=.4]

\node at (0,0) (a) {$1234$};

\node at (-8,5) (b1) {$1243$};
\node at (0,5) (b2) {$1324$};
\node at (8,5) (b3) {$2134$};

\node at (-16,10) (c1) {$1342$};
\node at (-8,10) (c2) {$1423$};
\node at (0,10) (c3) {$2143$};
\node at (8,10) (c4) {$2314$};
\node at (16,10) (c5) {$3124$};

\node at (-20,15) (d1) {$1432$};
\node at (-12,15) (d2) {$2341$};
\node at (-4,15) (d3) {$2413$};
\node at (4,15) (d4) {$3142$};
\node at (12,15) (d5) {$3214$};
\node at (20,15) (d6) {$4123$};

\node at (-16,20) (e1) {$2431$};
\node at (-8,20) (e2) {$3241$};
\node at (0,20) (e3) {$3412$};
\node at (8,20) (e4) {$4132$};
\node at (16,20) (e5) {$4213$};

\node at (-8,25) (f1) {$3421$};
\node at (0,25) (f2) {$4231$};
\node at (8,25) (f3) {$4312$};

\node at (0,30) (g) {$4321$};

\draw[-, double, thick, blue] (a) to (b1);
\draw[-, double, thick, blue] (a) to (b2);
\draw[-, double, thick, blue] (a) to (b3);

\draw[-, thick] (b1) to (c1);
\draw[-, thick] (b1) to (c2);
\draw[-, double, thick, blue] (b1) to (c3);

\draw[-, thick] (b2) to (c1);
\draw[-, thick] (b2) to (c2);
\draw[-, thick] (b2) to (c4);
\draw[-, thick] (b2) to (c5);

\draw[-, double, thick, blue] (b3) to (c3);
\draw[-, thick] (b3) to (c4);
\draw[-, thick] (b3) to (c5);

\draw[-, double, thick, blue] (c1) to (d1);
\draw[-, thick] (c1) to (d2);
\draw[-, thick] (c1) to (d4);

\draw[-, double, thick, blue] (c2) to (d1);
\draw[-, thick] (c2) to (d3);
\draw[-, thick] (c2) to (d6);

\draw[-, thick] (c3) to (d3);
\draw[-, thick] (c3) to (d4);

\draw[-, thick] (c4) to (d2);
\draw[-, thick] (c4) to (d3);
\draw[-, double, thick, blue] (c4) to (d5);

\draw[-, thick] (c5) to (d4);
\draw[-, double, thick, blue] (c5) to (d5);
\draw[-, thick] (c5) to (d6);

\draw[-, thick] (d1) to (e1);
\draw[-, thick] (d1) to (e4);

\draw[-, double, thick, blue] (d2) to (e1);
\draw[-, double, thick, blue] (d2) to (e2);

\draw[-, thick] (d3) to (e1);
\draw[-, thick] (d3) to (e3);
\draw[-, thick] (d3) to (e5);

\draw[-, thick] (d4) to (e2);
\draw[-, thick] (d4) to (e3);
\draw[-, thick] (d4) to (e4);

\draw[-, thick] (d5) to (e2);
\draw[-, thick] (d5) to (e5);

\draw[-, double, thick, blue] (d6) to (e4);
\draw[-, double, thick, blue] (d6) to (e5);

\draw[-, thick] (e1) to (f1);
\draw[-, thick] (e1) to (f2);
\draw[-, thick] (e2) to (f1);
\draw[-, thick] (e2) to (f2);

\draw[-, double, thick, blue] (e3) to (f1);
\draw[-, double, thick, blue] (e3) to (f3);

\draw[-, thick] (e4) to (f2);
\draw[-, thick] (e4) to (f3);

\draw[-, thick] (e5) to (f2);
\draw[-, thick] (e5) to (f3);

\draw[-, double, thick, blue] (f1) to (g);
\draw[-, double, thick, blue] (f2) to (g);
\draw[-, double, thick, blue] (f3) to (g);

\end{tikzpicture}
}
\caption{The two-sided weak order on $\S_4$.}\label{F:S4}
\end{center}
\end{figure}
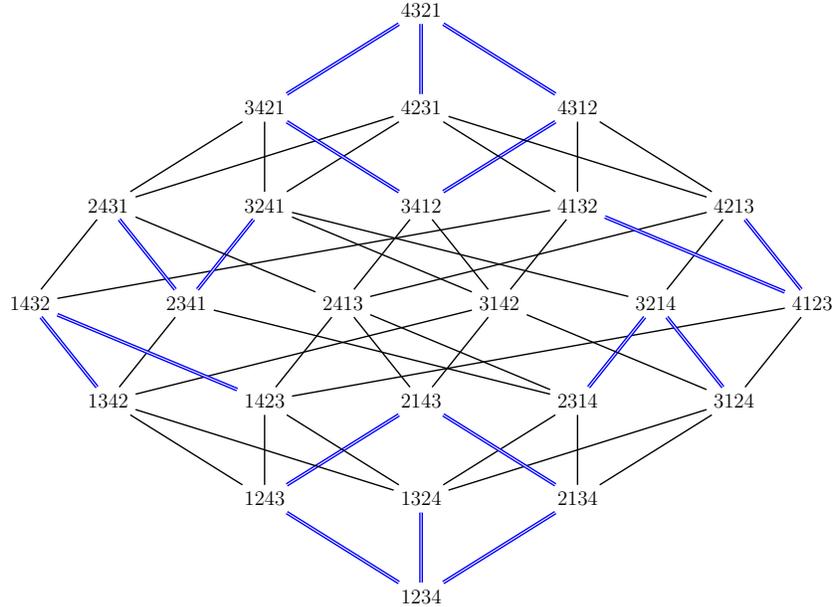

\end{example}

The first main result of this section is the following theorem.

\begin{theorem}\label{thm.weak-order-Sq}
The restriction of the weak order on the interval of clans
\[
[\gamma_e, \gamma_0]= \{\gamma_w \mid w\in \S_q\}
\]
is isomorphic to the two-sided weak order on $\S_q$. 
\end{theorem}

To begin, we prove that the restriction of the weak order to the interval $[\gamma_e, \gamma_0]$ is generated by only two of the cover relations described in Section~\ref{ss.weak-order} (cf.~Figure~\ref{F:coverrelations}).

\begin{lemma}\label{lemma.covers} Every cover relation of the weak order in the interval $[\gamma_e, \gamma_0]$ is of type IC1 or IC2.
\end{lemma}
\begin{proof} Let $\gamma_w=c_1c_2\cdots c_n \in \mathbf{Clan}_{p,q}$ for some $w\in \S_q$. 
We have by definition that
\[
c_1\cdots c_p = 12\cdots q+\cdots +,
\]
and furthermore that no $-$ signs occur in $\gamma_w$.
This implies that the cover relations of types IA1, IA2, and II do not occur in the restriction of the weak order on $\Clan_{p,q}$ to $[\gamma_e, \gamma_0]$.  Note also that no cover relation of type IB can occur among the clans in $[\gamma_e, \gamma_0]$ since all arcs are either nested or crossing as there is an arc connecting $i<j$ if and only if $j=w^{-1}(i)+p$. This finishes the proof of our assertion. 
\end{proof}

By the lemma, to analyze the weak order on $[\gamma_e, \gamma_0]$ it is enough to consider cover relations of type IC1 and IC2.  The following example illustrates a cover relation of each type.

\begin{example}
Let $p=6$ and $q=5$. 
Let $w=51324\in \S_5$. The following depicts the cover relation of type IC2 obtained by uncrossing the (dashed) arcs in the charged matching for $\gamma_w$ with right endpoints $8$ and $9$, creating a nested pair.  Note that the resulting matching corresponds to the clan $\gamma_{w s_2}$, and we have $w = 51324 <_R ws_2=53124$.

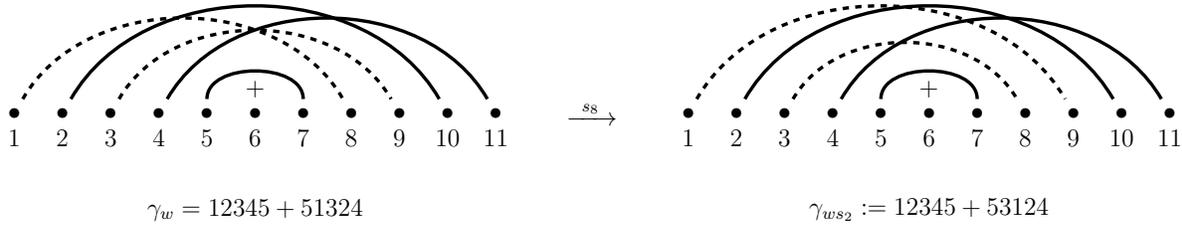
\begin{center}
\scalebox{.8}{
\begin{tikzpicture}[scale=.8]

\begin{scope}[xshift=7cm]
 \node (x1) at (-5,0) {$\bullet$}; 
 \node (x2) at (-4,0) {$\bullet$};
 \node (x3) at (-3,0) {$\bullet$};
 \node (x4) at (-2,0) {$\bullet$};
 \node (x5) at (-1,0) {$\bullet$};
 \node (x6) at (0,0) {$\bullet$};
 \node (x7) at (5,0) {$\bullet$};
 \node (x8) at (4,0) {$\bullet$};
 \node (x9) at (3,0) {$\bullet$};
 \node (x10) at (2,0) {$\bullet$};
 \node (x11) at (1,0) {$\bullet$};
  
 \node at (-5,-.5) {$1$}; 
 \node at (-4,-.5) {$2$};
 \node at (-3,-.5) {$3$};
 \node at (-2,-.5) {$4$};
  \node at (-1,-.5) {$5$};
\node at (0, -.5) {$6$};
 \node at (5,-.5) {$11$};
 \node at (4,-.5) {$10$};
 \node at (3,-.5) {$9$};
 \node at (2,-.5) {$8$};
  \node at (1,-.5) {$7$};

  \node at (0, .5) {$+$};
 
\draw [ultra thick] (x4) to[out=60 ,in=120] (x7) ;
\draw [ultra thick,dashed] (x1) to[out=60 ,in=120] (x9) ;
\draw [ultra thick,-] (x2) to[out=60 ,in=120] (x8) ;
\draw [ultra thick,dashed] (x3) to[out=60 ,in=120] (x10) ;
\draw [ultra thick] (x5) to[out=90 ,in=90] (x11) ;

\node at (0,-2) {$\gamma_{ws_2}:=12345+53124$};
\end{scope}

\begin{scope}[xshift=0]
\node at (0,0) {$\xrightarrow{\;\;s_8\;\;}$};
\end{scope}

\begin{scope}[xshift=-7cm]
 \node (x1) at (-5,0) {$\bullet$}; 
 \node (x2) at (-4,0) {$\bullet$};
 \node (x3) at (-3,0) {$\bullet$};
 \node (x4) at (-2,0) {$\bullet$};
 \node (x5) at (-1,0) {$\bullet$};
 \node (x6) at (0,0) {$\bullet$};
 \node (x7) at (5,0) {$\bullet$};
 \node (x8) at (4,0) {$\bullet$};
 \node (x9) at (3,0) {$\bullet$};
 \node (x10) at (2,0) {$\bullet$};
 \node (x11) at (1,0) {$\bullet$};

 \node at (-5,-.5) {$1$}; 
 \node at (-4,-.5) {$2$};
 \node at (-3,-.5) {$3$};
 \node at (-2,-.5) {$4$};
 \node at (-1,-.5) {$5$};
 \node at (0, -.5) {$6$};
 \node at (5,-.5) {$11$};
 \node at (4,-.5) {$10$};
 \node at (3,-.5) {$9$};
 \node at (2,-.5) {$8$};
 \node at (1,-.5) {$7$}; 

  \node at (0, .5) {$+$};
 
\draw [ultra thick] (x4) to[out=60 ,in=120] (x7) ;
\draw [ultra thick,dashed] (x1) to[out=60 ,in=120] (x10) ;
\draw [ultra thick] (x2) to[out=60 ,in=120] (x8) ;
\draw [ultra thick,dashed] (x3) to[out=60 ,in=120] (x9) ;
\draw [ultra thick] (x5) to[out=90 ,in=90] (x11) ;

\node at (0,-2) {$\gamma_w=12345+51324$};
\end{scope} 
 \end{tikzpicture}
}
\end{center}
Similarly, we may apply a cover relation of type IC1 to $\gamma_w$ by swapping the (dashed) arcs with left endpoints $1$ and $2$, creating a nested pair.  The resulting matching corresponds to clan $\gamma_{s_1w}$ and we have $w=51324<_L s_1w = 52314$. 
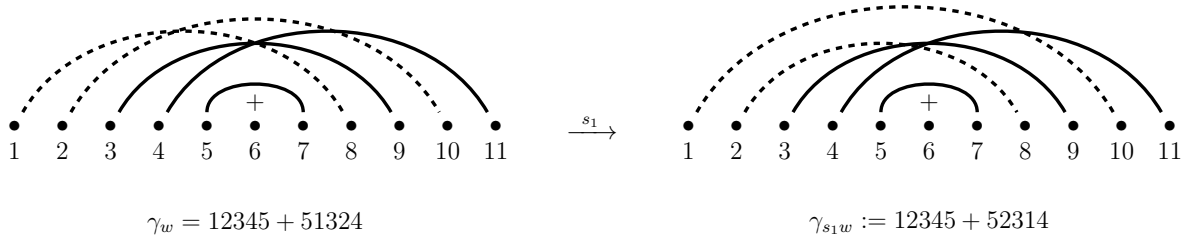
\begin{center}
\scalebox{.8}{
\begin{tikzpicture}[scale=.8]

\begin{scope}[xshift=7cm]
 \node (x1) at (-5,0) {$\bullet$}; 
 \node (x2) at (-4,0) {$\bullet$};
 \node (x3) at (-3,0) {$\bullet$};
 \node (x4) at (-2,0) {$\bullet$};
 \node (x5) at (-1,0) {$\bullet$};
 \node (x6) at (0,0) {$\bullet$};
 \node (x7) at (5,0) {$\bullet$};
 \node (x8) at (4,0) {$\bullet$};
 \node (x9) at (3,0) {$\bullet$};
 \node (x10) at (2,0) {$\bullet$};
 \node (x11) at (1,0) {$\bullet$};
  
 \node at (-5,-.5) {$1$}; 
 \node at (-4,-.5) {$2$};
 \node at (-3,-.5) {$3$};
 \node at (-2,-.5) {$4$};
  \node at (-1,-.5) {$5$};
\node at (0, -.5) {$6$};
 \node at (5,-.5) {$11$};
 \node at (4,-.5) {$10$};
 \node at (3,-.5) {$9$};
 \node at (2,-.5) {$8$};
  \node at (1,-.5) {$7$};

  \node at (0, .5) {$+$};
 
\draw [ultra thick] (x4) to[out=60 ,in=120] (x7) ;
\draw [ultra thick,dashed] (x1) to[out=60 ,in=120] (x8) ;
\draw [ultra thick,dashed] (x2) to[out=60 ,in=120] (x10) ;
\draw [ultra thick] (x3) to[out=60 ,in=120] (x9) ;
\draw [ultra thick] (x5) to[out=90 ,in=90] (x11) ;

\node at (0,-2) {$\gamma_{s_1w}:=12345+52314$};
\end{scope}

\begin{scope}[xshift=0]
\node at (0,0) {$\xrightarrow{\;\;s_1\;\;}$};
\end{scope}

\begin{scope}[xshift=-7cm]
 \node (x1) at (-5,0) {$\bullet$}; 
 \node (x2) at (-4,0) {$\bullet$};
 \node (x3) at (-3,0) {$\bullet$};
 \node (x4) at (-2,0) {$\bullet$};
 \node (x5) at (-1,0) {$\bullet$};
 \node (x6) at (0,0) {$\bullet$};
 \node (x7) at (5,0) {$\bullet$};
 \node (x8) at (4,0) {$\bullet$};
 \node (x9) at (3,0) {$\bullet$};
 \node (x10) at (2,0) {$\bullet$};
 \node (x11) at (1,0) {$\bullet$};

 \node at (-5,-.5) {$1$}; 
 \node at (-4,-.5) {$2$};
 \node at (-3,-.5) {$3$};
 \node at (-2,-.5) {$4$};
 \node at (-1,-.5) {$5$};
 \node at (0, -.5) {$6$};
 \node at (5,-.5) {$11$};
 \node at (4,-.5) {$10$};
 \node at (3,-.5) {$9$};
 \node at (2,-.5) {$8$};
 \node at (1,-.5) {$7$}; 

  \node at (0, .5) {$+$};
 
\draw [ultra thick] (x4) to[out=60 ,in=120] (x7) ;
\draw [ultra thick,dashed] (x1) to[out=60 ,in=120] (x10) ;
\draw [ultra thick,dashed] (x2) to[out=60 ,in=120] (x8) ;
\draw [ultra thick] (x3) to[out=60 ,in=120] (x9) ;
\draw [ultra thick] (x5) to[out=90 ,in=90] (x11) ;

\node at (0,-2) {$\gamma_w=12345+51324$};
\end{scope} 
 \end{tikzpicture}
}
\end{center}
\end{example}

In the example above, we saw that each covering relation was of the form $\gamma_w \prec \gamma_{w'}$ for $w,w'\in \S_5$ such that $w \prec w'$ in the two-sided weak order on $\S_5$.  This holds in greater generality and brings us to the proof of Theorem~\ref{thm.weak-order-Sq}.

\begin{proof}[Proof of Theorem~\ref{thm.weak-order-Sq}]

By Lemma~\ref{lemma.covers}, the covering relations of the weak order on $[\gamma_e,\gamma_0]$ are given by either a Type IC1 covering relation or by a Type IC2 covering relation. 

Now, a covering relation of Type IC1 on clans in $[\gamma_e, \gamma_0]$ is of the form 
\begin{align}\label{A:IC1holds}
\gamma_w = c_1 \cdots c_{w^{-1}(i)+p} \cdots c_{w^{-1}(i+1)+p} \cdots c_n \xrightarrow{\; s_i \;} \gamma' = c_1 \cdots c_{w^{-1}(i+1)+p} \cdots c_{w^{-1}(i)+p} \cdots c_n,
\end{align}
for some $i\in [q-1]$ such that $w^{-1}(i) < w^{-1}(i+1)$. Since the resulting clan is obtained from $\gamma_w$ by interchanging $i$ and $i+1$ in the one-line notation for $w$, we see that $\gamma' = \gamma_{s_i w}$.
Similarly, a covering relation of type IC2 on clans in $[\gamma_e, \gamma_0]$ is of the form
\begin{align}\label{A:IC2holds}
\gamma_w = c_1 \cdots c_ic_{i+1} \cdots c_n  \xrightarrow{\;s_i\;} \gamma'=c_1 \cdots  c_{i+1} c_i \cdots c_n,
\end{align}
for some $i\in \{p+1,\dots, n-1\}$ such that $w(i-p)<w(i-p+1)$. In this case, we have $\gamma' = \gamma_{ ws_{i-p}}$ since the resulting clan is obtained by interchanging the entries in the positions $i-p$ and $i-p+1$ in the one-line notation for $w$.
In conclusion, we see that, for $w,v\in \S_q$,
if the clan $\gamma_w$ is covered by the clan $\gamma_v$ in the weak order, then $w$ is covered by $v$ in either the right weak order or the left weak order on $\S_q$.

We proceed to prove the converse statement. 
Let $w,v\in \S_q$ be two permutations.
Let $\gamma_w=c_1^wc_2^w \cdots c_n^w$ denote the (unique) clan corresponding to $w$, which is defined by
\begin{itemize}
\item $c_i^w= +$ for all $q<i \leq p$, and
\item $c_i^w=c_{p+w^{-1}(i)}^w=i$ for all $i\in [q]$.
\end{itemize}
Let $\gamma_v$ denote the unique clan corresponding to $v$, defined in a similar manner.
Now, we assume that $w$ is covered by $v$ in the left weak order on $\S_q$. Hence, $s_iw= v$ holds for some $i\in [q-1]$.
After writing $w$ and $s_iw$ in their one-line notations, we see that the covering relation $w\leq_L v$ corresponds to the covering relation in (\ref{A:IC1holds}).
Likewise, 
if $w$ is covered by $v$ in the right weak order in such a way that $ws_i= v$ for some $i\in [q-1]$, then the covering relation in (\ref{A:IC2holds}) holds.
Hence, we proved that $\gamma_w \preceq \gamma_v$ if and only if $w\leq_L v$ or $w\leq_R v$, as desired.
\end{proof}

\begin{corollary}\label{cor.231-avoid} The restriction of the weak order to $\Clan_{p,q}^{231}$ is isomorphic to the restriction
of the two-sided weak order on $\S_q$ to all $231$-free permutations.  
\end{corollary}

\begin{remark}
It is well known that the right weak order on the set of $231$-free permutations 
is isomorphic to the Tamari lattice~\cite[Theorem 1.2]{Drake}.
It is also well-known that the Bruhat (i.e., inclusion) order on the set of $231$-free permutations is isomorphic to the opposite of the Dyck path lattice~\cite{BBFP}. 
\end{remark}

\begin{example} Let $p=q=3$. Figure~\ref{F:33} shows the weak order on $[\gamma_{123}, \gamma_{321}] \subset \Clan_{3,3}$ with all covering relations and corresponding clan written underneath each charged matching. The circled matching corresponds to the clan $\gamma_{231}$. By removing this matching, we obtain the Hasse diagram of the two-sided weak order on $\mathbf{Clan}^{231}_{3,3}$.

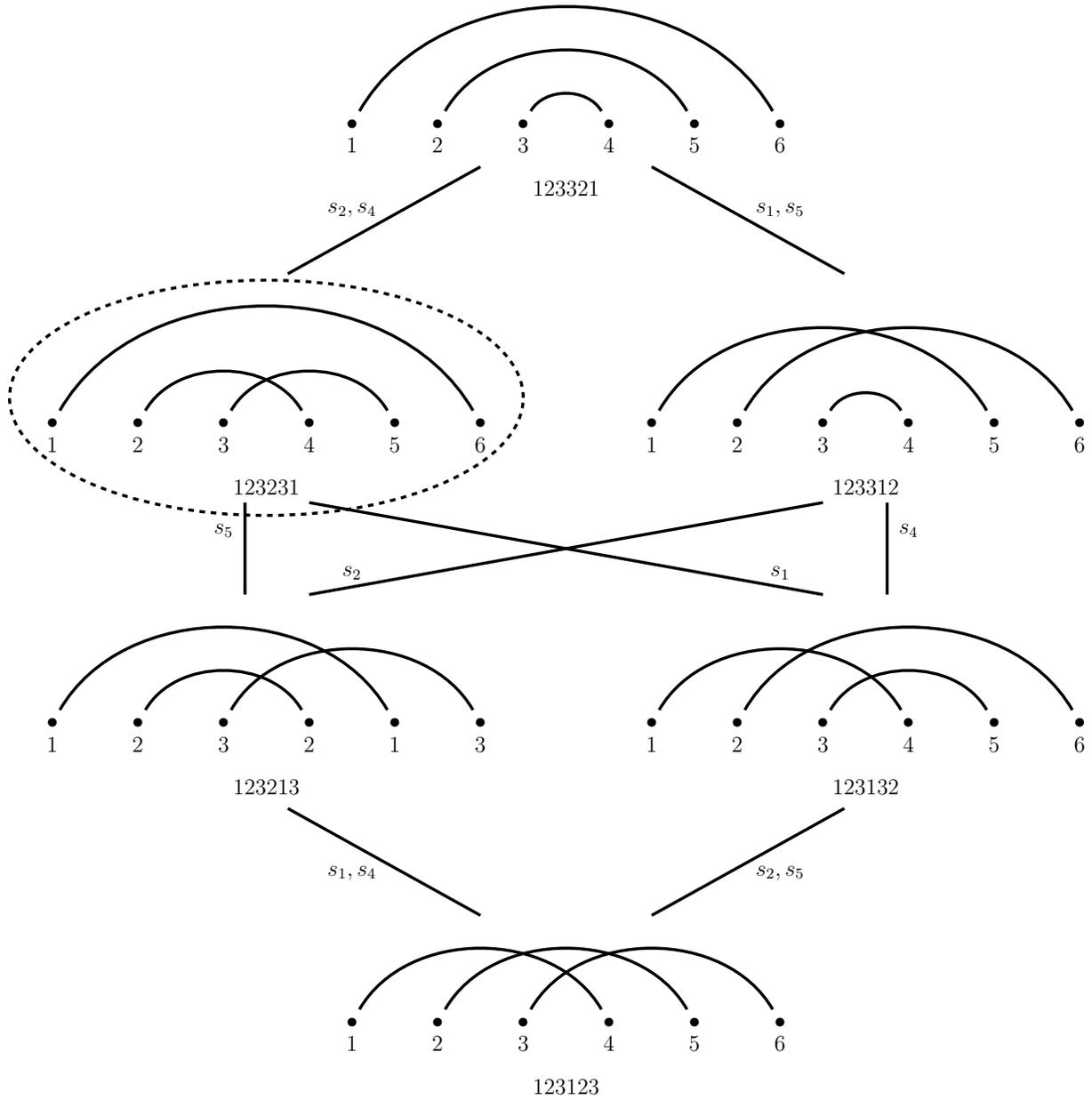
\begin{figure}[htp]
\begin{center}

\scalebox{.8}{
\begin{tikzpicture}[scale=.8]

\begin{scope} [xshift=0cm, yshift=-17cm]
 \node (a1) at (-5,0) {$\bullet$}; 
 \node (a2) at (-3,0) {$\bullet$};
 \node (a3) at (-1,0) {$\bullet$};
 \node (a4) at (1,0) {$\bullet$};
 \node (a5) at (3,0) {$\bullet$};
 \node (a6) at (5,0) {$\bullet$};
 \node at (-5,-.5) {$1$}; 
 \node at (-3,-.5) {$2$};
 \node at (-1,-.5) {$3$};
 \node at (1,-.5) {$4$};
 \node at (3,-.5) {$5$};
 \node at (5,-.5) {$6$};
\draw [ultra thick] (a1) to[out=60 ,in=120] (a4) ;
\draw [ultra thick] (a2) to[out=60 ,in=120] (a5) ;
\draw [ultra thick] (a3) to[out=60 ,in=120] (a6) ;
\node at (0,-1.5) {$123123$};
\end{scope} 

\begin{scope} [xshift=-7cm, yshift=-10cm]
 \node (a1) at (-5,0) {$\bullet$}; 
 \node (a2) at (-3,0) {$\bullet$};
 \node (a3) at (-1,0) {$\bullet$};
 \node (a4) at (1,0) {$\bullet$};
 \node (a5) at (3,0) {$\bullet$};
 \node (a6) at (5,0) {$\bullet$};
 \node at (-5,-.5) {$1$}; 
 \node at (-3,-.5) {$2$};
 \node at (-1,-.5) {$3$};
 \node at (1,-.5) {$2$};
 \node at (3,-.5) {$1$};
 \node at (5,-.5) {$3$};
\draw [ultra thick] (a1) to[out=60 ,in=120] (a5) ;
\draw [ultra thick] (a2) to[out=60 ,in=120] (a4) ;
\draw [ultra thick] (a3) to[out=60 ,in=120] (a6) ;
\node at (0,-1.5) {$123213$};
\end{scope}

\begin{scope} [xshift=7cm, yshift=-10cm]
 \node (a1) at (-5,0) {$\bullet$}; 
 \node (a2) at (-3,0) {$\bullet$};
 \node (a3) at (-1,0) {$\bullet$};
 \node (a4) at (1,0) {$\bullet$};
 \node (a5) at (3,0) {$\bullet$};
 \node (a6) at (5,0) {$\bullet$};
 \node at (-5,-.5) {$1$}; 
 \node at (-3,-.5) {$2$};
 \node at (-1,-.5) {$3$};
 \node at (1,-.5) {$4$};
 \node at (3,-.5) {$5$};
 \node at (5,-.5) {$6$};
\draw [ultra thick] (a1) to[out=60 ,in=120] (a4) ;
\draw [ultra thick] (a2) to[out=60 ,in=120] (a6) ;
\draw [ultra thick] (a3) to[out=60 ,in=120] (a5) ;
\node at (0,-1.5) {$123132$};
\end{scope} 

\begin{scope} [xshift=-7cm, yshift=-3cm]
 \node (a1) at (-5,0) {$\bullet$}; 
 \node (a2) at (-3,0) {$\bullet$};
 \node (a3) at (-1,0) {$\bullet$};
 \node (a4) at (1,0) {$\bullet$};
 \node (a5) at (3,0) {$\bullet$};
 \node (a6) at (5,0) {$\bullet$};
 \node at (-5,-.5) {$1$}; 
 \node at (-3,-.5) {$2$};
 \node at (-1,-.5) {$3$};
 \node at (1,-.5) {$4$};
 \node at (3,-.5) {$5$};
 \node at (5,-.5) {$6$};
\draw [ultra thick] (a1) to[out=60 ,in=120] (a6) ;
\draw [ultra thick] (a2) to[out=60 ,in=120] (a4) ;
\draw [ultra thick] (a3) to[out=60 ,in=120] (a5) ;
\draw[ultra thick, dashed] (0,.6) ellipse (6cm and 2.75cm);
\node at (0,-1.5) {$123231$};
\end{scope} 

\begin{scope} [xshift=7cm, yshift=-3cm]
 \node (a1) at (-5,0) {$\bullet$}; 
 \node (a2) at (-3,0) {$\bullet$};
 \node (a3) at (-1,0) {$\bullet$};
 \node (a4) at (1,0) {$\bullet$};
 \node (a5) at (3,0) {$\bullet$};
 \node (a6) at (5,0) {$\bullet$};
 \node at (-5,-.5) {$1$}; 
 \node at (-3,-.5) {$2$};
 \node at (-1,-.5) {$3$};
 \node at (1,-.5) {$4$};
 \node at (3,-.5) {$5$};
 \node at (5,-.5) {$6$};
\draw [ultra thick] (a1) to[out=60 ,in=120] (a5) ;
\draw [ultra thick] (a2) to[out=60 ,in=120] (a6) ;
\draw [ultra thick] (a3) to[out=60 ,in=120] (a4) ;
\node at (0,-1.5) {$123312$};
\end{scope} 

\begin{scope} [xshift=0cm, yshift=4cm]
 \node (a1) at (-5,0) {$\bullet$}; 
 \node (a2) at (-3,0) {$\bullet$};
 \node (a3) at (-1,0) {$\bullet$};
 \node (a4) at (1,0) {$\bullet$};
 \node (a5) at (3,0) {$\bullet$};
 \node (a6) at (5,0) {$\bullet$};
 \node at (-5,-.5) {$1$}; 
 \node at (-3,-.5) {$2$};
 \node at (-1,-.5) {$3$};
 \node at (1,-.5) {$4$};
 \node at (3,-.5) {$5$};
 \node at (5,-.5) {$6$};
\draw [ultra thick] (a1) to[out=60 ,in=120] (a6) ;
\draw [ultra thick] (a2) to[out=60 ,in=120] (a5) ;
\draw [ultra thick] (a3) to[out=60 ,in=120] (a4) ;
\node at (0,-1.5) {$123321$};
\end{scope}

\begin{scope}
\draw[ultra thick] (-2,-14.5) -- (-6.5,-12);
\node at (-5,-13.5) {$s_1,s_4$};
\end{scope}

\begin{scope}
\draw[ultra thick] (2,-14.5) -- (6.5,-12);
\node at (5,-13.5) {$s_2,s_5$};
\end{scope}

\begin{scope}
\draw[ultra thick] (-7.5,-7) -- (-7.5,-4.85);
\node at (-8,-5.5) {$s_5$};
\end{scope}

\begin{scope}
\draw[ultra thick] (7.5,-7) -- (7.5,-4.85);
\node at (8,-5.5) {$s_4$};
\end{scope}

\begin{scope}
\draw[ultra thick] (-6.5,.5) -- (-2,3);
\node at (-5,2) {$s_2, s_4$};
\end{scope}

\begin{scope}
\draw[ultra thick] (6.5,.5) -- (2,3);
\node at (5,2) {$s_1,s_5$};
\end{scope}

\begin{scope}
\draw[ultra thick] (-6,-7) -- (6,-4.85);
\draw[ultra thick] (6,-7) -- (-6,-4.85);
\node at (5,-6.5) {$s_1$};
\node at (-5,-6.5) {$s_2$};
\end{scope}

 \end{tikzpicture}
}
\end{center}
\caption{Weak order graph on $[\gamma_{123}, \gamma_{321}] \subset \Clan_{3,3}$.}
\label{F:33}
\end{figure}
\end{example}

With a precise description of the weak order on $[\gamma_e, \gamma_0] \subset \Clan_{p,q}$ in hand, we turn our attention to computing the $W$-sets $W(\gamma_w)$ for each $w\in \S_q$.  Using Brion's Theorem~\ref{thm.Brion}, we can use this set to obtain polynomial representatives of the cohomology classes $[\ov{\O_{\gamma_w}}]$ for each $w\in \S_q$. If $w$ avoids $231$ then we obtain, by our work in the previous section, a polynomial representative for the cohomology class of the semisimple Hessenberg variety $\Hess(\x_{p,q}, \m_w)$.

We begin with a lemma whose proof is evident. 

\begin{lemma}\label{L:Marthasmap}
The map $\varphi : \mathbf{S}_{n}\to \mathbf{S}_{n}$ defined by $\varphi(v) = w_0 v^{-1} w_0$ for all $v\in \mathbf{S}_{n}$ is an anti-involution.
In other words, we have $\varphi^2 = id$ and for all $v, w\in \mathbf{S}_{n}$ we have $\varphi(vw) = \varphi(w)\varphi(v)$. 
\end{lemma}

Since the map $\varphi$ defined in Lemma~\ref{L:Marthasmap} is an anti-involution, it is a bijection. 
We are interested in the restriction of $\varphi$ to the subgroup $\mathbf{S}_q:= \langle s_1,\dots, s_{q-1}\rangle \hookrightarrow \mathbf{S}_{n}$. 
Recall that the support $\Supp(w)$ of $w$ is the set of all simple reflections that arise in any reduced word for~$w$.  

\begin{lemma}\label{lem.phi.Sq} The restriction of $\varphi$ to $\S_q$ induces a bijection $\varphi: \S_q \to \left< s_{p+1}, \ldots, s_{n-1} \right>$. Furthermore, $\ell(v) = \ell(\varphi(v))$ for all $v\in S_q$ and $\Supp(u)\cap \Supp(\varphi(v)) = \emptyset$ for all $u,v\in \S_q$.
\end{lemma}
\begin{proof} Since $\varphi(s_i) = w_0 s_i w_0 = s_{n-i}$ for all $i\in [n-1]$, the first assertion of the lemma is obvious.  Next, if $s_{i_1}\cdots s_{i_r}$ is a reduced expression for $v\in \mathbf{S}_{q}$, then $s_{n-i_1}\cdots s_{n-i_r}$ is a reduced expression of $w_0 v w_0$. 
In particular, 
\[
v = s_{i_1}\cdots s_{i_r} \Rightarrow \varphi(v) = s_{n-i_{r}}s_{n-i_{r-1}}\cdots s_{n-i_1}
\]
so $\ell(\varphi(v))=\ell(v)$. Finally, $\Supp(u)\subseteq \{s_1, \ldots, s_{{q-1}}\}$ and $\Supp(\varphi(v)) \subseteq \{s_{p+1}, \ldots, s_{n-1}\}$.  Since $q\leq p$, we obtain the final assertion.
\end{proof}

We these observations in place, we define a map that will allow us to compute $W(\gamma_w)$ in Theorem~\ref{thm.W-set} below.

\begin{lemma}\label{lemma.key-map} The map
\begin{eqnarray}\label{eqn.key-map}
\S_q \times \S_q \to \S_n ,\; (u,v) \mapsto u\varphi(v)
\end{eqnarray}
is injective.  Furthermore, $\ell(u\varphi(v)) = \ell(u)+ \ell(v)$ for all $u,v\in \S_q$.
\end{lemma}
\begin{proof} Recall from Lemma~\ref{lem.phi.Sq} that $\varphi$ maps $\S_q = \left< s_1, \ldots, s_{q-1} \right>$ to  $\left< s_{p+1}, \ldots, s_{n-1} \right>$ and note that the intersection of these subgroups is the trivial group since $q\leq p$.  Thus, if $u_1,u_2,v_1,v_2\in \S_q$ such that $u_1\varphi(v_1) = u_2\varphi(v_2)$ then
\begin{align*}
u_2^{-1}u_1 &= \varphi(v_2)\varphi(v_1)^{-1}     \in  \S_q \cap \left< s_{p+1}, \ldots, s_{n-1} \right> = \left<e\right> \\
& \quad\quad\Rightarrow u_1=u_2 \;\textup{ and }\; \varphi(v_1)=\varphi(v_2),
\end{align*}
and injectivity of the map follows. Since $\Supp(u)\cap \Supp(\varphi(v))=\emptyset$, any reduced word for $u\varphi(v)$ is the product of a reduced word of $u$ in $\S_q$ and a reduced word for $\varphi(v)$ in $\left< s_{p+1}, \ldots, s_{n-1} \right>$. Thus $\ell(u\varphi(v)) = \ell(u)+\ell(\varphi(v)) = \ell(u)+\ell(v)$ as desired.
\end{proof}

For each $w\in \S_q$ we define the set 
\[
\mathcal{S}(w):= \{ (u,v)\in \mathbf{S}_q \times \mathbf{S}_q \mid\ w=uv \text{ and } \ell(w) = \ell(u)+\ell(v) \}.
\]
Note that $\mathcal{S}(e)=\{(e,e)\}$ and $\mathcal{S}(s_i) = \{(e,s_i), (s_i, e)\}$.

\begin{example}\label{ex.Sw-set}
If $q=3$ and $w=312 = s_2s_1$ then $\mathcal{S}(w) = \{ (e, s_2s_1), (s_2, s_1), (s_2s_1, e)  \}$.
\end{example}

Recall that $y_0\in \S_q$ denotes the longest element. The second main theorem of this section describes the $W$-sets of clans $\gamma_w$ concretely using the set $\mathcal{S}(wy_0)$.

\begin{theorem}\label{thm.W-set}
For all $w\in\S_q$ the restriction of the map~\eqref{eqn.key-map} from Lemma~\ref{lemma.key-map} to $\mathcal{S}(wy_0)\subseteq \S_q\times \S_q$ induces a bijection
\[
\psi_w: \mathcal{S}(wy_0) \to W(\gamma_w) ,\; \psi_w(u,v) := u\varphi(v).
\]
In particular, the $W$-set of the clan $\gamma_w$ is $W(\gamma_w) = \{ u\varphi(v) \mid (u,v)\in \mathcal{S}(wy_0)\}$.
\end{theorem}
\begin{proof}
We argue first that $\psi_w(u,v) \in W(\gamma_w)$ for all $(u,v)\in \mathcal{S}(w y_0)$. Given $(u,v)\in \mathcal{S}(w y_0)$, let $u=s_{a_1}s_{a_2}\cdots s_{a_r}$ and $v=s_{b_1}s_{b_2}\cdots s_{b_t}$ be reduced words for $u$ and $v$, respectively. By assumption,
\[
wy_0 =s_{a_1}s_{a_2}\cdots s_{a_r} s_{b_1}s_{b_2}\cdots s_{b_t}
\]
is a reduced word for $wy_0$.  Manipulating this expression and using the fact that $y_0 s_i y_0 = s_{q-i}$ for all $i$ implies
\[
s_{a_r} \cdots s_{a_2}s_{a_1} w s_{q-b_t}\cdots s_{q-b_2}s_{q-b_1} = y_0
\]
with $\ell(y_0) = \ell(w) + r+t$.
In particular, this expression yields a chain of length $r+t$ in the two-sided weak order on~$\S_q$: 
\begin{align*}
w \xrightarrow{ s_{a_1} } s_{a_1} w &\xrightarrow{\; s_{a_2} \; } s_{a_2}s_{a_1} w \rightarrow \cdots \xrightarrow{\; s_{a_r}\;} s_{a_r} \cdots s_{a_2}s_{a_1} w = u^{-1}w \\ & \xrightarrow{\; s_{q-b_t} \;} u^{-1}w  s_{q-b_t} \rightarrow  \cdots \xrightarrow{\; s_{q-b_2} \;}    u^{-1}w s_{q-b_t}\cdots s_{q-b_2} \xrightarrow{\; s_{q-b_1} \;} y_0. 
\end{align*}
In this chain, left multiplication by $s_{a_k}$ is a cover in the left weak order on $\S_q$ and corresponds to a cover of type IC1 on clans. This cover of type IC1 on clans is labeled by the simple reflection $s_{a_k}\in \S_n$. Right multiplication by $s_{b_k}$ is a cover in the right weak order on $\S_q$ and corresponds to a cover of type IC2 on clans. This cover of type IC2 on clans is labeled by the simple reflection $s_{n-b_k}= \varphi(s_{b_k}) \in \S_n$. 
By Theorem~\ref{thm.weak-order-Sq} and definition of the $W$-set, it follows that 
\[
u\varphi(v) = s_{a_1} s_{a_2 } \cdots s_{a_r} s_{n-b_t} \cdots s_{n-b_2}s_{n-b_1} \in W(\gamma_w) 
\]
as desired.

To complete the proof, it suffices by Lemma~\ref{lemma.key-map} to show that $\psi_w$ is surjective. We proceed by induction on the nonnegative integer $\ell(wy_0)=\ell(y_0)- \ell(w)$. If $\ell(wy_0)=0$ then $w=y_0$, $W(\gamma_w) = W(\gamma_0) = \{e\}$, and $\mathcal{S}(wy_0) = \mathcal{S}(e) = \{(e,e)\}$. Thus our claim holds trivially in this case. 

Suppose now that $w\in \S_q$ such that $\ell = \ell(wy_0)>0$ and $\psi_{w'}$ is surjective for all $w'\in \S_q$ such that $\ell(w'y_0) = \ell-1$. Since the $W$-set of $\gamma_w$ is obtained by multiplying the labels of the weak order cover relations along a saturated path from $\gamma_w$ to $\gamma_0$, if $x\in W(\gamma_w)$ there exists $w' \in \S_q$ and $x'\in W(\gamma_{w'})$ such that  $\gamma_w \xrightarrow{\;s_i\;} \gamma_{w'}$ and $x=s_i x'$. By Theorem~\ref{thm.weak-order-Sq}, $w'$ is a cover of $w$ in the two-sided weak order on $\S_q$ so $\ell(w') = \ell(w)+1$.  This in turn implies $\ell(w'y_0) = \ell(wy_0)-1 = \ell-1$ and the induction hypothesis implies that there exists $(u,v)\in \mathcal{S}(w'y_0)$ such that $x'=u\varphi(v)$.

There are two possible cases to consider: the cover $\gamma_w \prec \gamma_{w'}$ is either of type IC1 or IC2. If $\gamma_w \xrightarrow{\;s_i \;} \gamma_{w'}$ is a cover in the weak order on clans of type IC1, then the proof of Theorem~\ref{thm.weak-order-Sq} implies $i\in [q-1]$ and $w'=s_iw$. Our assumptions also yield
\begin{align}
\nonumber\ell(x) = \ell(x')+1 &\Rightarrow \ell(s_iu\varphi(v) ) = \ell(u\varphi(v))+1 \\
\nonumber&\Rightarrow \ell(s_i u) + \ell(v) = \ell(u) + \ell(v)+1 \\
\label{eqn.u.length}&\Rightarrow \ell(s_i u) = \ell(u)+1,
\end{align}
where the second implication follows from Lemma~\ref{lemma.key-map}.
This shows $(s_iu,v)\in \mathcal{S}(wy_0)$ since $s_i u v= s_i w'y_0 = wy_0$ and 
\[
\ell(w'y_0) = \ell(u)+\ell(v) \Rightarrow \ell (y_0) - \ell(w) -1  =  \ell(u)+\ell(v) \Rightarrow \ell(wy_0) = \ell(s_iu) + \ell(v)
\]
by~\eqref{eqn.u.length} above. Now $x= s_i u \varphi(v) = \psi_w (s_i u,v)$, so $\psi_w$ is surjective in this case.

If $\gamma_w \xrightarrow{\;s_i \;} \gamma_{w'}$ is a cover in the weak order on clans of type IC2, then the proof of Theorem~\ref{thm.weak-order-Sq} implies $i\in \{p+1, \ldots, n-1\}$ and $w' = w s_{i-p}$ with $\ell(w')$.  Note that $s_i$ commutes with $u\in \S_q$ and recall that $s_{i} = \varphi (s_{n-i})$. Thus 
\begin{eqnarray}\label{eqn.x}
x = s_i x'=s_i u \varphi(v) = u s_i \varphi(v) = u \varphi(s_{n-i}) \varphi(v) = u \varphi(vs_{n-i})
\end{eqnarray}
by Lemma~\ref{L:Marthasmap}. Our assumptions also imply 
\begin{align}
\nonumber\ell(x) = \ell(x')+1 &\Rightarrow \ell(u \varphi(vs_{n-i})) = \ell(u\varphi(v))+1 \\
\nonumber&\Rightarrow \ell(u) + \ell(vs_{n-i}) = \ell(u) + \ell(v)+1 \\
\label{eqn.v.length}&\Rightarrow \ell(vs_{n-i}) = \ell(v)+1,
\end{align}
where the first implication follows from~\eqref{eqn.x} and the second from Lemma~\ref{lemma.key-map}.
This shows $(u,vs_{n-i})\in \mathcal{S}(wy_0)$ since $uvs_{n-i} = w' y_0 s_{n-i} = w' s_{i-p} y_0 = wy_0$ and 
\[
\ell(w'y_0)= \ell(u)+\ell(v) \Rightarrow  \ell(y_0)-\ell(w)-1 = \ell(u)+\ell(v) \Rightarrow \ell(wy_0) = \ell(u) + \ell(vs_{n-i})
\]
by~\eqref{eqn.v.length} above.  Using~\eqref{eqn.x} we conclude $x=\psi_w(u, vs_{n-1})$ so $\psi_w$ is indeed surjective.
\end{proof}

\begin{example} Let $q=4$ and $w=3214 = s_1s_2s_1\in \S_4$. Then $wy_0 = 4123 = s_3s_2s_1$ and 
\[
\mathcal{S}(wy_0) = \mathcal{S}(s_3s_2s_1) = \{ (s_3s_2s_1,e), (s_3s_2, s_1),(s_3, s_2s_1), (e,s_3s_2s_1) \},
\]
so, according to Theorem~\ref{thm.W-set}, the $W$-set of $\gamma_{3214}$ is 
\[
\{s_3s_2s_1, s_3s_2 s_{3+p}, s_3s_{3+p}s_{2+p}, s_{3+p}s_{2+p}s_{1+p}\}.
\]
The interested reader can also confirm this using Theorem~\ref{thm.weak-order-Sq} and the poset pictured in Figure~\ref{F:S4}. Each element of the $W$-set is obtained from a saturated chain in the poset connecting $3214$ to $y_0=4321$.  Covers arising from the right weak order (respectively, left weak order) on $\S_q$ labeled by $s_i$ correspond to covers in the weak order on clans labeled by $s_{i+p}$ (respectively $s_i$).  Note that there are more chains than elements of the $W$-set, as two chains can yield the same reduced word.
\end{example}

We apply the results of Theorem~\ref{thm.W-set} to compute the cohomology class of each $K$-orbit closure $\overline{\O}_{\gamma_w}$. 
We make use of Borel's description of the integral cohomology ring $H^*(GL_n/B,\Z)$ as the ring of coinvariants, that is, 
\[
H^*(GL_n/B,\Z) \simeq \Z[ x_1,\dots, x_{n}]/I, 
\]
where $I$ is the ideal generated by the symmetric polynomials without a constant term. It is a well-known fact that the Schubert polynomial $\mf{S}_w$ is a polynomial representative for the cohomology class $[\overline{Bw_0wB/B}]$. For a more detailed definition of Schubert polynomials see~\cite{Lascoux-Schutzenberger, Manivel}. Combining Theorem~\ref{thm.W-set} with Brion's Theorem~\ref{thm.Brion} now yields the following.

\begin{proposition} \label{P:Schubertforgammaw} For all $w\in \S_q$, the cohomology class of the closure of the $K$-orbit $\O_{\gamma_w}$ is represented by the polynomial
\[
\mf{S}(\gamma_w):= \sum_{(u,v)\in \mathcal{S}(wy_0)} \mf{S}_{u\varphi(v)} = \sum_{(u,v)\in \mathcal{S}(wy_0)} \mf{S}_{u}\mf{S}_{\varphi(v)}. 
\]
\end{proposition}
\begin{proof} By Theorem~\ref{thm.Brion}, the polynomial representative of the cohomology class of $\ov{\O_{\gamma_w}}$ is given by the formula  
\begin{align*}
\mf{S}(\gamma_w) := \sum_{x\in W(\gamma_w)} \mf{S}_x,
\end{align*}
where $\mf{S}_x$ is the Schubert polynomial indexed by the permutation $x\in \S_{n}$. By Theorem~\ref{thm.W-set} each $x\in W(\gamma_w)$ can be written $x = u \varphi(v)$ for a unique $(u,v)\in \mathcal{S}(wy_0)$.  The result now follows immediately, as $u$ and $\varphi(v)$ have disjoint supports by Lemma~\ref{lem.phi.Sq}, so $\mf{S}_{u\varphi(v)} = \mf{S}_u \mf{S}_{\varphi(v)}$ (see, for example, \cite[Corollary~2.4.6]{Manivel}.)
\end{proof}

The following is now immediate from Theorem~\ref{thm.irreducible}.

\begin{corollary}\label{cor.Hess.class} For all $w\in \S_q$ avoiding the pattern $231$, the polynomial representative of the cohomology class of Hessenberg variety $\Hess(\x_{p,q},\m(w))$ is given by 
\[
\mf{S}(\Hess(\x_{p,q}, \m(w)))= \sum_{(u,v)\in \mathcal{S}(wy_0)} \mf{S}_{u}\mf{S}_{\varphi(v)}.
\]
\end{corollary}

\begin{example}\label{ex:cohomclass}
Let $q=3$. For the permutation $w=123 \in \S_3$, the $W$-set of $\gamma_{123}$ is 
\[\{s_1s_2s_1, s_{p+1}s_{p+2}s_{p+1}, s_1s_2s_{p+2}, s_2s_1s_{p+1}, s_1s_{p+2}s_{p+1}, s_2s_{p+1}s_{p+2} \}.\] It follows from Corollary~\ref{cor.Hess.class} that the polynomial representative of the cohomology class for the Hessenberg variety
$\Hess(\x_{p,3}, \m(123))$ is
\[
 \mf{S}_{s_1s_2s_1} + \mf{S}_{s_{p+1}s_{p+2}s_{p+1}} + \mf{S}_{s_1s_2s_{p+2}} + \mf{S}_{s_2s_1s_{p+1}} + \mf{S}_{s_1s_{p+2}s_{p+1}} + \mf{S}_{s_2s_{p+1}s_{p+2}}.
\]

Applying a similar calculation to the permutation $w= 213 \in \S_3$  gives us the polynomial
\[
\mf{S}_{s_2s_1} + \mf{S}_{s_2s_{p+2}}
+ \mf{S}_{s_{p+2}s_{p+1}}, \]
representing the cohomology class for the Hesseberg variety $\Hess(\x_{p,3}, \m(213))$,
as the $W$-set of $\gamma_{213}$ is $\{s_2s_1, s_2s_{p+2}, s_{p+2}s_{p+1}\}$.
\end{example}

Our final goal is to understand the intersection of the closure of the orbit corresponding to the clan $\gamma_w$ with a ``basic hyperplane of $G L_n / B$.'' Here, by a basic hyperplane of $G L_n / B$, we mean the Schubert divisor $X_{s_iw_{0}}$, and $i \in [n-1]$. Such an intersection is succinctly expressed in the cohomology ring by Monk's formula~\cite[Theorem~2.7.1]{Manivel}.

\begin{lemma}[Monk's formula] For all $u \in \mathbf{S}_n$ and all $m \in[n-1]$,
$$
\mathfrak{S}_{s_m} \mathfrak{S}_u =\sum_{\substack{j \leq m<k \\ \ell\left(u t_{j k}\right)=\ell(u)+1}} \mathfrak{S}_{u t_{j k}},
$$
where $t_{j k}$ is the transposition in $\mathbf{S}_n$ that interchanges $j$ and $k$ and leaves every other number fixed.

\end{lemma}

\begin{example} Let $q=3$.
We know from Example~\ref{ex:cohomclass} that the cohomology class for the Hessenberg variety $\Hess(\x_{p,3}, \m(123))$ represented by 
\[
\mf{S}_{s_1s_2s_1} + \mf{S}_{s_{p+1}s_{p+2}s_{p+1}} + \mf{S}_{s_1s_2s_{p+2}} + \mf{S}_{s_2s_1s_{p+1}} + \mf{S}_{s_1s_{p+2}s_{p+1}} + \mf{S}_{s_2s_{p+1}s_{p+2}}.
\]
Now, let us use Monk's formula to understand the product $\mathfrak{S}_{s_m}\mathfrak{S}({\gamma_{123}})$ for all $m< n=6$. We note that in each case, the product is a $0-1$ sum of Schubert polynomials.

Multiplying by $\mf{S}_{s_1}$ gives us 
\begin{align*}
\mf{S}_{s_3s_1s_2s_1} &+ \mf{S}_{s_1s_{p+1}s_{p+2}s_{p+1}} + \mf{S}_{s_1s_2s_1s_{p+2}} + \mf{S}_{s_{p+1}s_3s_2s_1} +\\ &\mf{S}_{s_3s_2s_1s_{p+1}} + \mf{S}_{s_2s_1s_{p+2}s_{p+1}} +\mf{S}_{s_2s_1s_{p+1}s_{p+2}} + \mf{S}_{s_1s_2s_{p+1}s_{p+2}}.
\end{align*}

Multiplying by $\mf{S}_{s_2}$ gives us 
\begin{align*}
\mf{S}_{s_3s_1s_2s_1} &+ \mf{S}_{s_2s_3s_1s_2} + \mf{S}_{s_2s_{p+1}s_{p+2}s_{p+1}} + \mf{S}_{s_3s_1s_2s_{p+2}} + \\ 
&\mf{S}_{s_{p+1}s_3s_2s_1} + \mf{S}_{s_3s_2s_1s_{p+1}} + \mf{S}_{s_1s_2s_1s_{p+1}} +\mf{S}_{s_2s_1s_{p+2}s_{p+1}} +\\
&\mf{S}_{s_1s_2s_{p+2}s_{p+1}} + \mf{S}_{s_1s_2s_{p+1}s_{p+2}} + \mf{S}_{s_{p+1}s_{p+2}s_3s_2} + \mf{S}_{s_3s_2s_{p+1}s_{p+2}}.
\end{align*}

Multiplying by $\mf{S}_{s_3}$ gives us 
\begin{align*}
\mf{S}_{s_2s_1s_3s_2} &+\mf{S}_{s_1s_2s_1s_3} + \mf{S}_{s_{p+1}s_{p+2}s_{p+1}s_3} + \mf{S}_{s_{p+1}s_3s_{p+2}s_{p+1}} + \\ &\mf{S}_{s_3s_{p+1}s_{p+2}s_{p+1}} + 
\mf{S}_{s_3s_1s_2s_{p+2}} + \mf{S}_{s_1s_2s_3s_{p+2}} + \mf{S}_{s_{p+1}s_3s_2s_1} +\\
&\mf{S}_{s_3s_{p+1}s_2s_1} + \mf{S}_{s_{p+1}s_2s_1s_3} + \mf{S}_{s_2s_1s_3s_{p+1}} + 
\mf{S}_{s_{p+2}s_{p+1}s_1s_3} + \\ &\mf{S}_{s_3s_1s_{p+2}s_{p+1}} + \mf{S}_{s_3s_2s_{p+1}s_{p+2}} + \mf{S}_{s_2s_{p+1}s_{p+2}s_3} + \mf{S}_{s_2s_3s_{p+1}s_{p+2}}.
\end{align*}

Multiplying by $\mf{S}_{s_4}$ gives us 
\begin{align*}
\mf{S}_{s_{p+1}s_1s_2s_1} &+ \mf{S}_{s_{p+1}s_{p+2}s_3s_{p+1}} + \mf{S}_{s_3s_{p+1}s_{p+2}s_{p+1}} + \mf{S}_{s_1s_2s_{p+2}s_{p+1}} + \\ 
&\mf{S}_{s_1s_2s_{p+1}s_{p+2}} + \mf{S}_{s_3s_2s_1s_{p+1}} + \mf{S}_{s_2s_1s_3s_{p+1}} +\mf{S}_{s_2s_1s_{p+2}s_{p+1}} +\\
&\mf{S}_{s_3s_1s_{p+2}s_{p+1}} + \mf{S}_{s_3s_2s_{p+1}s_{p+2}} + \mf{S}_{s_2s_3s_{p+1}s_{p+2}} + \mf{S}_{s_1s_{p+1}s_{p+2}s_{p+1}}.
\end{align*}

Multiplying by $\mf{S}_{s_5}$ gives us 
\begin{align*}
\mf{S}_{s_3s_{p+1}s_{p+2}s_{p+1}} &+ \mf{S}_{s_{p+2}s_1s_2s_1} + \mf{S}_{s_1s_2s_{p+1}s_{p+2}} + \mf{S}_{s_2s_1s_{p+2}s_{p+1}} +\\ &\mf{S}_{s_2s_1s_{p+1}s_{p+2}} + \mf{S}_{s_1s_{p+1}s_{p+2}s_{p+1}} +\mf{S}_{s_3s_2s_{p+1}s_{p+2}} + \mf{S}_{s_2s_3s_{p+1}s_{p+2}}.
\end{align*}
\end{example}

\medskip

We can use Monk's formula to understand the product $\mathfrak{S}_{s_m}\mathfrak{S}({\gamma_w})$ in general.
In particular, we show that the product $\mf{S}_{s_m}\mf{S}({\gamma_w})$ is a $0-1$ sum of Schubert polynomials for all $m < n$.

\begin{theorem} \label{thm.monkmultfree}
If $m \in [n-1]$ and $w\in \S_q$, then the product $\mathfrak{S}_{s_m}\mathfrak{S}({\gamma_w})$ is a multiplicity-free sum of Schubert polynomials. 
\end{theorem}

\begin{proof} 
It follows from Proposition~\ref{P:Schubertforgammaw} that
\begin{eqnarray}\label{eqn.prod1}
\mf{S}_{s_m} \mathfrak{S}({\gamma_w})=  \sum_{(u,v) \in \mathcal{S}(wy_0)}  \mathfrak{S}_{s_m} \mf{S}_{u\varphi(v)}.
\end{eqnarray}
We apply Monk's formula to each product $\mf{S}_{s_m}\mf{S}_{u \varphi(v)}$ and obtain
\begin{align}\label{eqn.prod2}
\mf{S}_{s_m}\mf{S}_{u\varphi(v)} =\displaystyle \sum_{\substack{j \leq m<k \\ \ell\left(u\varphi(v) t_{j k}\right)=\ell(u)+\ell(v)+1}} \mf{S}_{u\varphi(v)t_{j k}}.
\end{align}
Suppose there exists pairs 
$(u_1,v_1)$ and $(u_2,v_2)$ in $\mathcal{S}(w)$ such that 
\begin{eqnarray}\label{eqn.eq}
u_1\varphi(v_1)t_{jk} = u_2\varphi(v_2)t_{j'k'}
\end{eqnarray}
for some $j,k,j',k'$ such that $j\leq m<k$ and $j'\leq m<k'$.  To complete the proof, we show that~\eqref{eqn.eq} implies $u_1=u_2$ and $v_1=v_2$. 

Write $x_1=u_1\varphi(v_1)$ and $x_2=u_2\varphi(v_2)$ for the remainder of the proof.  We begin with a few observations. By construction, each permutation $x_i=u_i\varphi(v_i)$ with $ i\in \{1,2\}$ satisfies 
\begin{eqnarray}\label{eqn.perm1}
[ q] =\{x_i(1), \ldots, x_i(q)\}
\end{eqnarray}
and 
\begin{eqnarray}\label{eqn.perm2}
x_i(a)=a \textup{ for all } q+1\leq a \leq p,
\end{eqnarray}
and 
\begin{eqnarray}\label{eqn.perm3}
[n]\setminus [p] = \{x_i(p+1), \ldots, x_i(n)\}.
\end{eqnarray}
We obtain the one-line notation of $x_1t_{jk}$ from that of $x_1$ by exchanging the entries in positions $j$ and $k$ and similarly for $x_2t_{j'k'}$. These observations imply that pairs $j<k$ and $j'<k'$ satisfying~\eqref{eqn.eq} must fall into one of the following cases:
\begin{enumerate}
\item $j,j'\in [q]$, $k,k'\in [n]\setminus [q]$, 
\item $j,j',k,k'\in [q]$,
\item $j,j',k,k'\in \{q+1, \ldots, p\}$,
\item $j,j'\in \{q+1, \ldots, p\}$ and $k,k'\in [n]\setminus [p]$, and
\item  $j,j',k,k'\in [n]\setminus [p] $.
\end{enumerate}
Note that cases (3) and (4) do not arise when $p=q$.  

\medskip
We begin with Case (1). 
In this case,  the equality~\eqref{eqn.eq} and equation~\eqref{eqn.perm1} imply
\begin{eqnarray*}
\{k\} &=& ([n]\setminus [q]) \cap \{(x_1t_{jk})^{-1}(1), \ldots,  (x_1t_{jk})^{-1}(q)\}\\ &=&([n]\setminus [q]) \cap  \{(x_2t_{j'k'})^{-1}(1), \ldots,  (x_2t_{j'k'})^{-1}(q)\} =  \{k'\}
\end{eqnarray*}
so $k= k'$.  Similarly, using~\eqref{eqn.eq}, \eqref{eqn.perm2}, and~\eqref{eqn.perm3} we have
\begin{eqnarray*}
\{j\} &=& [q] \cap \{(x_1t_{jk})^{-1}(q+1), \ldots,  (x_1t_{jk})^{-1}(n)\}\\ &=&[q] \cap  \{(x_2t_{j'k'})^{-1}(q+1), \ldots,  (x_2t_{j'k'})^{-1}(n)\} =  \{j'\}
\end{eqnarray*}
so $j=j'$. Now~\eqref{eqn.eq} implies $x_1=x_2$ and thus $u_1=u_2$ and $v_1=v_2$ by Lemma~\ref{lemma.key-map}.  Case (4) follows by similar reasoning, so we omit it to avoid repetition.

\medskip
Now suppose we are in the setting of Case (2). Then~\eqref{eqn.eq} becomes $u_1t_{jk}\varphi(v_1) = u_2t_{j'k'}\varphi(v_2)$ since $t_{jk}$ and $t_{j'k'}$ commute with $\varphi(v_1)$ and $\varphi(v_2)$. By Lemma~\ref{lem.phi.Sq}, we also know that both of the following sets
\[
\operatorname{Supp}(u_1t_{jk}) \cap \operatorname{Supp}(\varphi(v_1))\quad\text{and}\quad \operatorname{Supp}(u_2t_{j'k'})\cap \operatorname{Supp}(\varphi(v_2))
\]
are empty. Consequently, $\operatorname{Supp}(\varphi(v_1)) = \operatorname{Supp}(\varphi(v_2)) $ implying that $\varphi(v_1)= \varphi(v_2)$, and hence $v_1 = v_2$ by Lemma~\ref{L:Marthasmap}.
It now follows from the definition of the set $\mathcal{S}(wy_0)$ that $u_1=u_2$. Indeed, we have $$u_1v_1 = wy_0 = u_2v_2 \textup{ and } v_1=v_2 \Rightarrow u_1v_1 = u_2v_1 \Rightarrow u_1=u_2.$$
The proof of Case (5) is almost identical to that of (2), so we omit it to avoid repetition.

\medskip
Finally, we consider Case (3).
The equality~\eqref{eqn.eq} and equation~\eqref{eqn.perm2} immediately imply that $j=j'$ and $k=k'$. Thus $x_1=x_2$ and we conclude $u_1=u_2$ and $v_1=v_2$ as before. This finishes the proof of our theorem.
\end{proof}


\end{document}